\documentclass[10pt]{amsart}
\usepackage[latin1]{inputenc}
 \usepackage[T1]{fontenc}
 \usepackage{lmodern}
\usepackage{amsfonts}
\usepackage{amsmath}

\usepackage[all]{xy} 
\usepackage{verbatim}
\usepackage{tikz, graphics}

\newtheorem{thm}{Theorem}[section]
\newtheorem{cor}[thm]{Corollary}
\newtheorem{prop}[thm]{Proposition}
\newtheorem{lem}[thm]{Lemma}

\theoremstyle{definition}
\newtheorem{defn}[thm]{Definition}

\newtheorem{rmk}[thm]{Remark}

\newtheorem*{ack}{Acknowledgments}

\numberwithin{equation}{section}

\newcommand{\U}{\mathcal U}
\newcommand{\ot}{\otimes}

\newcommand{\gl}{\mathfrak{gl}}

\newcommand{\R}{\mathbb R}
\newcommand{\C}{\mathbb C}
\newcommand{\Z}{\mathbb {Z}}
\newcommand{\N}{\mathbb {N}}

\newcommand{\NN}{\mathcal N}
\newcommand{\Om}{\Omega}

\newcommand*\mycirc[1]{%
  \begin{tikzpicture}
    \node[draw,circle,inner sep=.5pt] {#1};
  \end{tikzpicture}}

\title{Loop differential K-theory}

\author[T.~Tradler]{Thomas~Tradler}
  \address{Thomas Tradler,
  Department of Mathematics, College of Technology, City University of New York, 300 Jay Street, Brooklyn, NY 11201}
  \email{ttradler@citytech.cuny.edu}

\author[S.~Wilson]{Scott O. Wilson}
  \address{Scott O. Wilson, Department of Mathematics, Queens College, City University of New York, 65-30 Kissena Blvd., Flushing, NY 11367}
  \email{scott.wilson@qc.cuny.edu}

\author[M.~Zeinalian]{Mahmoud~Zeinalian}
  \address{Mahmoud Zeinalian, Department of Mathematics, Long Island University, LIU Post, 720 Northern Boulevard, Brookville, NY 11548, USA} 
  \email{mzeinalian@liu.edu}

\begin{document}
\maketitle

\begin{abstract}
In this paper we introduce an equivariant extension of the Chern-Simons form, associated to a path of connections on a bundle over a manifold $M$, to the free loop space $LM$, and show it determines an equivalence relation on the set of connections on a bundle. We use this to define a ring, \emph{loop differential K-theory of $M$}, in much the same way that differential K-theory can be defined using the Chern-Simons form \cite{SS}. 
We show loop differential K-theory yields a refinement of differential K-theory, and in particular incorporates holonomy information into its classes. Additionally, loop differential K-theory is shown to be strictly coarser than the Grothendieck group of bundles with connection up to gauge equivalence. Finally, we calculate loop differential K-theory of the circle.
\end{abstract}
\allowdisplaybreaks

\setcounter{tocdepth}{2}
\tableofcontents

\section{Introduction}

Much attention has been given recently to differential cohomology theories, as they play an increasingly important role in geometry, topology and mathematical physics. Intuitively these theories improve on classical (extra)-ordinary cohomology theories by including some additional cocycle information. Such differential cohomology theories have been shown abstractly to exist in \cite{HS}, and perhaps equally as important, they are often given by some differential-geometric representatives. This illuminates not only the mathematical theory, but also helps give mathematical meaning to several discussions in physics. For example, differential ordinary cohomology (in degree 2) codifies solutions to Maxwell's equations satisfying a Dirac quantization condition, while differential K-theory (and twisted versions) aids in explaining the Ramond-Ramond field in Type-II string theories \cite{FH}, \cite{FMS}. 
Additionally, it is expected that several differential cohomology theories can be described in terms of low dimensional topological field theories, using an appropriate notion of geometric concordance \cite{ST}.

Differential K-theory itself is a geometric enrichment of ordinary K-theory, having 
several formulations, \cite{HS, BS, L, SS}. It is known that differential $K$-theory is determined uniquely by a set axioms, first given in \cite{BS2}. For the purposes of this paper we focus on the even degree part of $K$-theory, denoted by $K^0$, and the model of even differential $K$-theory presented by Simons-Sullivan in \cite{SS}, which proceeds by defining
an equivalence relation on the set of a connections on a bundle, by requiring that the Chern-Simons form, associated to a path of connections, is exact. 
Elements in this presentation of differential $K$-theory contain the additional co-cycle information of a representative for the Chern character. 

In this paper, we show that a path of connections $\nabla_s$ in fact determines an odd differential form on the free loop space $LM$ of the base manifold $M$, which we denote by $BCS(\nabla_s)$ and we call the \emph{Bismut-Chern-Simons} form. 
When restricted to the base manifold along the constant loops, we obtain the ordinary Chern-Simons form, see Proposition 
 \ref{prop:restBCS}.
 Furthermore, this form satisfies the following fundamental homotopy formula  
\[
(d+\iota)(BCS(\nabla_s))=BCh(\nabla_1)-BCh(\nabla_0),
\]
where $\iota$ is the contraction by the natural vector field on $LM$ induced by the circle action, and $BCh(\nabla)$ is the 
Bismut-Chern form on $LM$, see Theorem \ref{thm:dBCS}.

Proceeding in much the same way as in \cite{SS}, we prove that the condition of $BCS(\nabla_s)$ being exact defines an equivalence relation on the set of connections on a bundle, and use this to define a functor from manifolds to rings, which we call loop differential K-theory. Elements in this ring contain the additional information of the trace of holonomy of a connection, and in fact the entire extension of the trace of holonomy to a co-cycle on the free loop space known as the Bismut-Chern form, which is an equivariantly closed form on the free loop space that restricts to the classical Chern character \cite{B}, \cite{GJP}, \cite{Ha}, \cite{TWZ}.

As we show, the loop differential K-theory functor, denoted by $M \mapsto L \widehat K^0(M)$, maps naturally to K-theory by a forgetful map $f$, forgetting the connection, and to even $(d+\iota)$-closed differential forms on $LM$, denoted $\Omega^{\textrm{even}}_{(d+\iota)-cl} (LM)$. This latter map, denoted $BCh$ for Bismut-Chern character, gives the following commutative diagram of ring homomorphisms:
\[
\xymatrix{
 & K^0(M) \ar[rd]^{[BCh]} & \\
 L \widehat K^0(M) \ar[ru]^f \ar [rd]^{BCh} & &H^{\textrm{even}}_{S^1}(LM) \\
 & \Omega^{\textrm{even}}_{(d+\iota)-cl} (LM) \ar[ru] &
}
\]
Here $H^{\textrm{even}}_{S^1}(LM)$ denotes (the even part) of the quotient of the kernel of $(d+\iota)$ by the image of $(d+\iota)$, where $d+ \iota$ is restricted to differential forms on $LM$ in the kernel of $d\iota + \iota d = (d+\iota)^2$. 
 
An analogous commutative diagram for (even) differential K-theory $\widehat K^0(M)$ was established in \cite{SS}, and in fact 
the commutative diagram above maps to this analogous square for differential K-theory, making the following commutative diagram of ring homomorphisms:
\[
\xymatrix{
 &  K^0(M) \ar[rd]_{[BCh]}  \ar@{=}[rrd]^{id}  & & & \\
L \widehat K^0(M) \ar[ru]^f \ar [rd]_{BCh} \ar[rrd]^{\pi} & & H^{\textrm{even}}_{S^1}(LM) \ar[rrd]_{\rho^*}  &  K^0(M) \ar[rd]^{[Ch]}  & \\
& \Omega^{\textrm{even}}_{(d+\iota)-cl} (LM)  \ar[ru]  \ar[rrd]_{\rho^*}  & \widehat K^0(M)  \ar[ru]^g \ar [rd]^{Ch}  & & H^{\textrm{even}}(M)\\
& & & \Omega^{\textrm{even}}_{d-cl} (M)  \ar[ru] &
}
\]
Here $H^{\textrm{even}}(M)$ denotes the deRham cohomology of $M$, $\rho^*$ is the restriction to constant loops, $\pi$ is a well defined surjective restriction map by Proposition \ref{prop:restBCS}, $g$ is the forgetful map, and $Ch$ is the classical Chern character.

In Corollary \ref{cor:LoopdKrefines} we show the map $\pi : L \widehat K^0(M) \to \widehat K^0(M)$ is in general not one-to-one. In fact, elementary geometric examples are constructed over the circle to explain the lack of injectivity, showing loop differential K-theory of the circle contains strictly more information than differential K-theory of the circle. On the other hand, we also show in Corollary \ref{cor:LoopdKcoarser} that loop differential K-theory is strictly coarser than the ring induced by all bundles with connection up to gauge equivalence. The situation is clarified by a diagram of implications in section \ref{sec:implications}. In the final section of the paper, we calculate the ring $L \widehat K^0(S^1)$. In short, elements of $L \widehat K^0(S^1)$ are determined by the spectrum of holonomy. 

Since the differential forms on $LM$  do not have nice locality properties with respect to $M$, one does not expect $L \widehat K^0(M)$ to have nice descent properties with respect to open sets of $M$. Giving prominence to such locality properties, Bunke, Nikolaus and V\"olkl construct in \cite{BNV} a theory $\widehat{{\bf ku}}_{loop}$ related to $L \widehat K^0(M)$  through sheafification. Such a construction makes drastic changes in favor of locality with respect to open sets of $M$. It is shown in \cite[Section 6.2]{BNV}, that there is a map $L \widehat K^0(M)\to \widehat{{\bf ku}}_{loop}(M)$, which, in general, is not injective. 

We close by emphasizing that the Bismut-Chern form, and many of the properties used herein, have  been given a field theoretic interpretation by Han, Stolz and Teichner. Namely, they can be understood in terms of dimensional reduction from a $1|1$ Euclidean field theory on $M$ to a $0|1$ Euclidean field theory on $LM$ \cite{Ha}, \cite{ST}. We are optimistic that the extension of the Chern-Simons form to the free loop space, referred to here as the Bismut-Chern-Simons form,
 will also have a field theoretic interpretation, and may also be of interest in other mathematical discussions that begin with the Chern-Simons form, such as 3-dimensional TFT's, quantum computation, and knot invariants. 

\begin{ack}
We would like to thank Dennis Sullivan, Stefan Stolz, Peter Teichner, and James Simons for useful conversations concerning the topics of this paper. We also thank Jim Stasheff for comments on an earlier draft, which helped to improve the paper. The authors were partially supported by the NSF grant DMS-0757245. The first and second authors were supported in part by grants from The City University of New York PSC-CUNY Research Award Program. The third author was partially supported by the NSF grant DMS-1309099 and would like to thank the Max Planck Institute for their support and hospitality during his visit.
\end{ack}

\section{The Chern and Chern-Simons Forms on $M$}\label{SEC:CS-on-M}

In this section we recall some basic facts about the Chern-Simons form on a manifold $M$, which is associated to a path of connections on a bundle over $M$.

\begin{defn} Given a connection $\nabla$ on a complex vector bundle $E \to M$,  with curvature $2$-form $R$, we define the Chern-Weil form by
\begin{equation} \label{eq:Ch}
Ch(\nabla):= Tr(\exp(R)) = Tr \left (\sum_{n\geq 0} \frac{1}{n!} \underbrace{R\wedge \dots\wedge R}_n  \right) \in \Omega^{even}(M)
\end{equation}
For a time dependent connection $\nabla_s$ we denote the Chern form at time $s$ by $Ch(\nabla_s)$.
\end{defn}

For a path of connections $\nabla_s$, $s \in [0,1]$, the Chern forms $Ch(\nabla_1)$ and $Ch(\nabla_0)$ are related by the
 odd Chern-Simons form $CS(\nabla_s) \in \Omega^{odd}(M)$ as follows.

\begin{defn}Let $\nabla_s$ be a path of connections on a complex vector bundle $E \to M$. The Chern-Simons form is given by
\begin{equation} \label{eq:CS}
CS(\nabla_s)= Tr \left( \int_0^1 \sum_{n\geq 1} \frac{1}{n!} \sum_{i=1}^n ( R_s\wedge \dots\wedge R_s\wedge \underbrace{\nabla'_s}_{i^{\text{th}}}\wedge R_s\wedge\dots\wedge R_s ) ds. \right)
\end{equation}
where $\nabla'_s=\frac{\partial}{\partial s} \nabla_s$.
\end{defn}

Since connections are an affine space modeled over the vector space of $1$-forms with values in 
$End( E)$,  the derivative $\nabla'_s$ lives in $\Om^1( M; End (E))$, so $CS(\nabla_s)$ is a well defined differential form on $M$. We note that the formula above agrees with another common presentation, where all the terms $\nabla'_s$ are brought to the front. The fundamental homotopy formula involving $CS(\nabla_s)$ is the following \cite{CS, SS}:
\begin{prop} \label{prop:dCS}
For a path of connections $\nabla_s$ we have:
\[
d(CS(\nabla_s)) =Ch(\nabla_1)-Ch(\nabla_0)
\]
\end{prop}

\section{The Bismut-Chern Form on $LM$}\label{SEC:BCh-on-LM}

Recall that the free loop space $LM$ of a smooth manifold $M$ is an infinite dimensional manifold, where the deRham complex  is well defined \cite{H}.
In fact much of this theory is not needed here as the differential forms we construct can all be expressed locally as iterated integrals of differential forms on the finite dimensional manifold $M$. 

The space $LM$ has a natural vector field, given by the circle action, whose induced contraction operator on differential forms is denoted by $\iota$.
Let $\Omega_{S^1}(LM) = \Omega^{\textrm{even}}_{S^1}(LM) \oplus  \Omega^{\textrm{odd}}_{S^1}(LM)$ denote 
 the $\Z_2$-graded differential graded algebra of forms on $LM$ in the kernel of $(d + \iota)^2 = d \iota + \iota d$,  with differential given by $(d+\iota)$. We let $H_{S^1}(LM) = H^{\textrm{even}}_{S^1}(LM) \oplus  H^{\textrm{odd}}_{S^1}(LM)$ denote the cohomology of $\Omega_{S^1}(LM)$ with respect to the differential $(d+\iota)$. Recall that this cohomology group can be computed completely in terms of the cohomology of $M$, see \cite{JP}.

We remark that the results which follow can also be restated in terms of the \emph{periodic complex} which is given by the operator $(d + u \iota)$ on the $\Z$-graded vector space $\Omega(LM)[u,u^{-1}]]$, consisting of Laurent series in $u^{-1}$, where $u$ has degree $2$.

Associated to each connection $\nabla$ on a complex vector bundle $E \to M$, 
there is an even form on the free loopspace $LM$ 
whose restriction to constant loops equals the Chern form $Ch(\nabla)$ of the connection. This result is due to Bismut, and so we refer to this form as the Bismut-Chern form on $LM$, and denote it by $BCh(E, \nabla)$, or $BCh(\nabla)$ if the context is clear. 

In \cite{TWZ} we gave an alternative construction where $BCh(E,\nabla)=\sum_{k\geq 0} Tr (hol_{2k})$ and
$Tr (hol_{2k}) \in \Omega^{2k}_{S^1}(LM)$. We now recall a local description of this. On any  single chart $U$ of $M$, we can write a connection locally as a matrix $A$ of 
$1$-forms, with curvature $R$, and in this case the restriction $Tr(hol^U_{2k})$  of $Tr(hol_{2k}) $ to $LU$ is given by
\begin{equation} \label{eq:BCh^U_{2k}}
Tr (hol^U_{2k}) = Tr \left( \sum_{m \geq k} \, \,  \sum_{1 \leq j_1 < \dots < j_k \leq m} \int_{\Delta^m} X_1(t_1) \cdots X_m(t_m) dt_1 \cdots dt_m \right) ,
\end{equation}
where
\[
X_j (t_j) = \left\{
\begin{array}{rl}
R(t_j) & \text{if  } j \in \{ j_1, \dots , j_k\} \\
\iota A(t_j)  & \text{otherwise}
\end{array} \right.
\]
Here $R(t_j)$ is a $2$-form taking in two vectors at $\gamma(t_j)$ on a loop $\gamma \in U$ , and 
$\iota A(t_j) = A (\gamma'(t_j))$. This defines a differential form on $LU$ since a tangent vector to a loop is a vector field along that loop, and we may evaluate the above expression by inserting the given vector fields at the prescribed times, and integrating.

Note that $Tr(hol_0)$ is the trace of the 
usual holonomy, and heustically $Tr(hol^U_{2k}) $ is given by the same formula for the trace of holonomy except with exactly $k$ copies of the function $\iota A$ replaced by the $2$-form $R$, summed over all possible places. Since the terms $X_j$ are smooth they have bounded values and derivatives, so this series converges for the same reason that holonomy itself converges; it is comparable to an exponential series. This same argument is used to justify the convergence of related series below.

More generally, a global form on $LM$ is defined as follows \cite{TWZ}.
We first remark that if $ \{U_i\}$ is a covering of $M$ then there is an induced covering of $LM$ in the following way. For any $p\in \mathbb N$, and $p$ open sets $\U=(U_{i_1},\dots, U_{i_p})$ from the cover $\{U_i\}$, there is an induced open subset $\NN(p,\U)\subset LM$ given by
\[
\NN(p,\U)=\left\{\gamma\in LM: \left(\gamma\Big|_{\big[\frac{k-1}{p},\frac{k}{p}\big]}\right)\subset U_{i_k}, \forall k=1,\dots,p \right\}.
\]
By the Lebesgue lemma, the collection $\{\NN(p,\U)\}_{p,i_1,\dots,i_p}$ forms an open cover of $LM$.
  
We fix a covering $\{U_i\}$ of $M$ over which we have trivialized $E|_{U_i} \to U_i$, and write the connection locally as a matrix valued $1$-form $A_i$ on $U_i$, with curvature $R_i$.
For a given loop $\gamma \in LM$ we can choose sets $\U = \{U_1, \ldots , U_p\}$ that cover a subdivision of $\gamma$ into 
$p$ the subintervals $[(k-1)/p, k/p]$, using a formula like \eqref{eq:BCh^U_{2k}}  on the open sets $U_j$ together with the transition functions $g_{i,j}: U_i \cap U_j \to Gl(n,\C)$ on overlaps. 
Concretely, we have
\begin{defn} \label{defn:BCh}
For $k \geq 0$, $Tr(hol_{2k}^{(p,\U)}) \in \Omega^{2k}(LM)$ is given by 
\begin{multline} \label{eq:BCh}
Tr(hol_{2k}^{(p,\U)}) \\ = Tr \Bigg (\sum_{n_1,\dots,n_p\geq 0} \quad \sum_{\scriptsize
\begin{matrix}
J\subset S\\
|J| = k
\end{matrix}
}\quad 
g_{i_p,i_1}
\wedge \bigg( \int_{\Delta^{n_1}} X^1_{i_1}\left(\frac {t_1} p\right) \cdots X^{n_1}_{i_1}\left(\frac {t_{n_1}} p\right) dt_1 \cdots dt_{n_1} \bigg)
\\
\wedge  g_{i_1,i_2}
\cdots   g_{i_{p-1},i_p}
\wedge \bigg(
\int_{\Delta^{n_p}} X^1_{i_p}\left(\frac {p- 1 + t_1} p \right) \cdots X^{n_p}_{i_p}\left(\frac {p- 1+t_{n_p}} p\right) dt_1 \cdots dt_{n_p} 
\bigg) \Bigg)
\end{multline}
where $g_{i_{k-1},i_k}$ is evaluated at $\gamma((k-1)/p)$, and the second sum is a sum over all $k$-element index sets $J\subset S$ of the sets
$S= \{(i_r,j): r=1,\dots, p,\text{ and } 1\leq j\leq n_r\}$, and
\[
X^j_i = \left\{
\begin{array}{rl}
R_i & \text{if  } (i,j) \in J\\
\iota A_i & \text{otherwise.}
\end{array} \right.
\] 
\end{defn}

Note that $Tr(hol_{0}^{(p,\U)})$ is precisely the trace of holonomy, and that heuristically $Tr(hol_{2k}^{(p,\U)})$  is this same formula for the trace of holonomy but with $k$ copies of $R$ shuffled throughout.

In \cite{TWZ} it is shown that $Tr(hol_{2k}^{(p,\U)})$ is independent of covering $(p,\U)$ and trivializations of $E \to M$, and
so defines a global form $Tr(hol_{2k})$ on $LM$. The techniques are repeated in Appendix \ref{app:A}.
Moreover, it is shown that these differential forms $Tr(hol_{2k})$ satisfy the fundamental property
\[
d Tr(hol_{2k}) = -\iota_{d/dt} Tr(hol_{2(k+1)}) \quad \quad \textrm{for all} \quad k \geq 0,
\]
where $d/dt$ is the canonical vector field on $LM$ given by rotating the circle. The Bismut-Chern form is then given by 
\[
BCh(\nabla)= \sum_{k \geq 0} Tr(hol_{2k}) \in \Om^{even}_{S^1}(LM),
\]
and it follows from the above that $(d+ \iota) BCh(\nabla) = 0$ and $(d\iota + \iota d) BCh(\nabla) =0$, where we abbreviate $\iota=\iota_{d/dt}$. Therefore, $BCh(\nabla) $ determines a class $[BCh(\nabla)]$ in the equivariant cohomology $H^{even}_{S^1} (LM)$, known as the Bismut-Chern class. It is shown in \cite{Z} that this class is in fact independent of the connection $\nabla$ chosen. 
An independent proof of this fact will  be given in the next section (Corollary \ref{Cor:BChcharacter}), using a lifting of the Chern-Simons form on $M$ to $LM$.

\begin{prop} \label{prop:restBCh}
For any connection $\nabla$ on a complex vector bundle $E \to M$, 
\[
\rho^* BCh(\nabla) = Ch(\nabla)
\]
where $\rho^* :\Om_{S^1} (LM) \to \Om(M)$ is the restriction to constant loops.
\end{prop}

\begin{proof}
Consider the restriction of formula \eqref{eq:BCh}  to $M$, for any $p$ and $\U$ . Since the local forms $\iota A$ vanish on constant loops, the only non-zero integrands are those that contain only $R$. Now, $R$ is globally defined on $M$, as a form with values in $End(E)$, so we may take $p=1$ and $\U = \{M\}$ for the definition of $Tr(hol_{2k})(\nabla)$. In this case, the formula for $Tr(hol_{2k})^{(p,\U)}$ agrees with the Chern form in \eqref{eq:Ch} since $1/n!$ is the volume of the $n$-simplex.
\end{proof}

The following proposition gives the fundamental properties of the Bismut-Chern form with respect to direct sums and tensor products. 
By restricting to constant loops, or instead to degree zero, one obtains the corresponding results which are known to hold for both the ordinary Chern form, and the trace of holonomy, respectively. In fact, we regard the proposition below as a hybridization of these  two deducible facts.

\begin{thm} \label{thm:BChsumtensor}
Let $(E,\nabla) \to M$ and $(\bar E, \bar \nabla) \to M$ be complex vector bundles with connections. 
Let $\nabla \oplus \bar \nabla$ be the induced connections on $E \oplus \bar E \to M$, and 
$\nabla \otimes \bar \nabla := \nabla \otimes Id + Id  \otimes \bar \nabla$ be the induced connection on $E \otimes \bar E \to M$. Then
\[
BCh(\nabla \oplus \bar \nabla) = BCh(\nabla) + BCh(\bar \nabla)
\]
and
\[
BCh(\nabla \ot \bar \nabla) = BCh(\nabla) \wedge BCh(\bar \nabla)
\]
\end{thm}

\begin{proof} We may assume that $E$ and $\bar E$ are locally trivialized over a common covering $\{U_i\}$ with transition functions $g_{ij}$ and $h_{ij}$, respectively. If $\nabla$ and $\bar \nabla$ are locally represented by $A_i$ and $B_i$ on $U_i$, then $\nabla \oplus \bar  \nabla$ is locally given by the block matrixes with blocks  $A_i$ and $B_i$. Similarly, this holds for transition functions and curvatures. The result now follows from Definition \ref{defn:BCh}, since block matrices are a subalgebra, and trace is additive along blocks.

For the second statement, it suffices to show that for all $k \geq0$ 
\begin{equation} \label{holtensor1}
Tr\left( hol_{2k}(\nabla \ot \bar \nabla)\right) = \sum_{\tiny\begin{matrix}{i+ j = k}\\{i,j \geq 0}\end{matrix}} Tr(hol_{2i} (\nabla)) \cdot Tr( hol_{2j}(\bar \nabla)) 
\end{equation} 
Note for $k=0$ this is just the well known fact that trace of holonomy is multiplicative.
If we express $\nabla$ and $\bar \nabla$  locally by $A_i$ and $B_i$ on $U_i$, then
$\nabla \ot \bar \nabla$ is locally given by $A_i \ot Id + Id \ot B_i$. Similarly,
the curvature is $R_i \ot Id + Id \ot S_i$, if $R_i$ and $S_i$ are the curvatures of $A_i$ and $B_i$, respectively.

We calculate $Tr\left( hol_{2k}(\nabla \ot \bar \nabla)\right) $ directly from Definition \ref{defn:BCh} using coordinate transition functions $g_{ij} \ot h_{ij}$:
\begin{multline*}
Tr \Bigg (\sum_{n_1,\dots,n_p\geq 0} \quad \sum_{\scriptsize
\begin{matrix}
J\subset S\\
|J| = k
\end{matrix}
}\quad g_{i_p,i_1} \ot h_{i_p,i_1}
\\
\wedge \bigg( \int_{\Delta^{n_1}} X^1_{i_1}\left(\frac {t_1} p\right) \cdots X^{n_1}_{i_1}\left(\frac {t_{n_1}} p\right) dt_1 \cdots dt_{n_1} \bigg)
\wedge  g_{i_1,i_2} \ot h_{i_1,i_2}
\\
\cdots   g_{i_{p-1},i_p} \ot h_{i_{p-1},i_p}
\wedge \bigg(
\int_{\Delta^{n_p}} X^1_{i_p}\left(\frac {p- 1 + t_1} p \right) \cdots X^{n_p}_{i_p}\left(\frac {p- 1+t_{n_p}} p\right) dt_1 \cdots dt_{n_p} 
\bigg) \Bigg)
\end{multline*}
where $g_{i_{k-1},i_k}$ is evaluated at $\gamma((k-1)/p)$, and the second sum is a sum over all $k$-element index sets $J\subset S$ of the sets
$S= \{(i_r,j): r=1,\dots, p,\text{ and } 1\leq j\leq n_r\}$, and
\[
X^j_i = \left\{
\begin{array}{rl}
R_i \ot Id + Id \ot S_i & \text{if  } (i,j) \in J\\
\iota A_i \ot Id + Id \ot \iota B_i & \text{otherwise.}
\end{array} \right.
\]

On each neighborhood $U_i$ above, for each choice of $m=n_j$ and $\ell 
\leq m$, we can apply the fact that
\[
 \sum_{\tiny\begin{matrix}{K \subset S_m}\\{|K|=\ell}\end{matrix}} X^1\left( t_1 \right) \cdots X^m \left( t_m \right) 
\quad \textrm{where} \quad
X^i = \left\{
\begin{array}{rl}
R \ot Id + Id \ot S & \text{if  } i \in K\\
\iota A \ot Id + Id \ot \iota B & \text{otherwise.}
\end{array} \right.
\]
for $S_m = \{1, \ldots , m\}$, is equal to 
\[
= \sum_{m_1+m_2=m} \sum_{\tiny\begin{matrix}{T_{m_1} \subset S_m}\\{ |T_{m_1} | = m_1} \end{matrix}} 
\sum_{\tiny\begin{matrix}{K_1 \subset T_{m_1} , K_2 \subset S_m- T_{m_1}}\\{|K_1| + |K_2| = \ell}\end{matrix}} \left(
Y^{\alpha_1} \cdots Y^{\alpha_{m_1}} \right) \ot \left(Y^{\beta_1} \cdots Y^{\beta_{m_2}}  \right) 
\]
where 
\[
Y^{\alpha_i} = \left\{
\begin{array}{rl}
R(t_{\alpha_i}) & \text{if  } \alpha_i \in K_1 \\
\iota A (t_{\alpha_i})& \alpha_i \in T_{m_1} - K_1
\end{array} \right.
\quad \quad
Y^{\beta_i} = \left\{
\begin{array}{rl}
S(t_{\beta_i}) & \text{if  } \beta_i \in K_2 \\
\iota B (t_{\beta_i})& \beta_i \in (S_m-T_{m_1}) - K_2
\end{array} \right.
\]
Now, integrating this expression over $\Delta^m$ and combining this integral with the sum over $T_{m_1}\subset S_m$ with $|T_{m_1}|=m_1$, we see that this becomes an integral over $\bigcup_{\tiny\begin{matrix}{T_{m_1} \subset S_m}\\{ |T_{m_1} | = m_1} \end{matrix}}  \Delta^m=\Delta^{m_1}\times \Delta^{m_2}$, where for $T_{m_1}=\{\alpha_1<\dots<\alpha_{m_1}\}$ and $S_m-T_{m_1}=\{\beta_1<\dots <\beta_{m_2}\}$ we use the inclusion $\Delta^m\hookrightarrow \Delta^{m_1}\times \Delta^{m_2}, (t_1\leq\dots \leq t_m)\mapsto ((t_{\alpha_1}\leq \dots\leq t_{\alpha_{m_1} }),(t_{\beta_1}\leq \dots\leq t_{\beta_{m_2} }))$ and we use the fact that these inclusions only intersect on lower dimensional faces. We therefore see that 
\begin{multline*}
\int_{\Delta^m} \sum_{\tiny\begin{matrix}{K \subset S_m}\\{|K|=\ell}\end{matrix}} X^1\left( t_1 \right) \cdots X^m \left( t_m \right) \\ = 
\sum_{m_1+m_2=m} \,\,\,
\sum_{\tiny\begin{matrix}{K_1 \subset S_{m_1} , K_2 \subset S_{m_2} }\\{|K_1| + |K_2| = \ell}\end{matrix}} \left(
\int_{\Delta^{m_1}}  Y^{1} \cdots Y^{m_1} \right) \ot \left( \int_{\Delta^{m_2}} Z^{1} \cdots Z^{m_2}  \right) 
\end{multline*}
where
\[
Y^{i} = \left\{
\begin{array}{rl}
R(t_{i}) & \text{if  } i \in K_1 \\
\iota A (t_{i})& i \in S_{m_1} - K_1
\end{array} \right.
\quad \quad
Z^{i} = \left\{
\begin{array}{rl}
S(t_{i}) & \text{if  } i \in K_2 \\
\iota B (t_i)& i \in S_{m_2} - K_2
\end{array} \right.
\]
By multi-linearity, this shows
\[
hol_{2k}(\nabla \ot \bar \nabla) \\ = \sum_{\tiny\begin{matrix}{i+ j = k}\\{i,j \geq 0}\end{matrix}} hol_{2i} (\nabla) \ot  hol_{2j}(\bar \nabla)
\]
Then \eqref{holtensor1} follows by taking trace of both sides, since $Tr(X \ot Y) = Tr(X)Tr(Y)$.
\end{proof}

\section{The Bismut-Chern-Simons Form on $LM$}\label{SEC:CS-on-LM}

Using a similar setup and collection of ideas as in the previous section, we construct for each path of connections on a complex 
vector bundle $E \to M$, an odd form on $LM$ which interpolates between the two Bismut-Chern forms of the endpoints of the path. Similarly to the presentation for $BCh$ above, we begin with a local discussion. 

Let $A_s$ with $s\in [0,1]$ be a path of connections on a single chart $U$ of $M$, with curvature $R_s$.
We let $A'_s=\frac{\partial A_s}{\partial s}$ and $R'_s=\frac{\partial R_s}{\partial s}$. For each $k \geq 0$, we define the following degree $2k+1$ differential form on $LU$, 
\begin{multline} \label{eq:BCS^U_{2k+1}}
BCS^{U}_{2k+1} (A_s) = Tr\Bigg( \sum_{n\geq k+1}\,\,\sum_{1\leq j_1<\dots<j_k\leq n}\,\,\sum_{\tiny\begin{matrix}{r=1}\\{r\neq j_1,\dots,j_k}\end{matrix}}^n
\\
 \int_0^1 \int_{\Delta^n}
 \iota A_s(t_1)\dots R_s(t_{j_1})
 \dots A'_s(t_r)\dots R_s (t_{j_k})\dots \iota A_s(t_n) \quad dt_1\dots dt_n ds 
  \Bigg)
 \end{multline}
 Here there is exactly one $A'_s$ at $t_r$, and there are exactly $k$ wedge products of $R_s$ at positions $t_{j_1},\dots,t_{j_k}\neq t_r$, and the remaining factors are $\iota A_s$. 
Heuristically, \eqref{eq:BCS^U_{2k+1}} is similar to \eqref{eq:BCh^U_{2k}}, except there is exactly one $A'_s$, summed over all possible times $t_r$, and integrated over $s=0$ to $s=1$. This formula can be understood in terms of iterated integrals, just as $BCh(\nabla)$ was understood in \cite{GJP} and \cite{TWZ}. 
 It is evident that the restriction of this form to $U$ equals the degree $2k+1$ part of the Chern-Simons form on $U$ since $\iota A$ vanishes on  constant loops, and the volume of the $n$-simplex is $1/n!$.
 
 More generally, we define an odd form on $LM$ as follows.
 Let $ \{U_i\}$ be a covering of $M$ over which we have trivialized $E|_{U_i} \to U_i$, with the connection given locally as a matrix valued $1$-form $A_i$ on $U_i$, with curvature $R_i$. Let $\{\NN(p,\U)\}_{p,i_1,\dots,i_p}$ be the induced cover of $LM$, as in the previous  section. For a given loop $\gamma \in LM$, we can choose sets $\U = \{U_1, \ldots , U_p\}$ that cover a subdivision of $\gamma$ into $p$ subintervals, and then use a formula like \eqref{eq:BCS^U_{2k+1}} on
the open sets $U_i$, and multiply these together (in order) by the transition functions $g_{i,j}: U_i \cap U_j \to Gl(n,\C)$.
Concretely, we have

\begin{defn} \label{defn:BCS}
Let $E \to M$ be a complex vector bundle. Let $\nabla_s$ be a path of connections on $E \to M$, and let $\U = \{U_i\}$ be a covering of $M$, with local trivializations of $E|_{U_i} \to U_i$. For these trivializations we write $\nabla_s$ locally as $A_{s,i}$ on $U_i$, with curvature $R_{s,i}$.
As before, we let $A'_{s,i}=\frac{\partial A_{s,i}}{\partial s}$, and $R'_{s,i}=\frac{\partial R_{s,i}}{\partial s}$.

For each $k \geq 0$, we define the following degree $2k+1$ differential form on $LM$, 
\begin{multline} \label{eq:BCS}
BCS^{(p,\U)}_{2k+1} 
= Tr \Bigg ( \int_0^1 \sum_{n_1,\dots,n_p\geq 0} \quad \sum_{\scriptsize
\begin{matrix}
J\subset S, |J| = k \\
(i_q,m) \in S - J
\end{matrix}
}\quad g_{i_p,i_1} 
\\
\wedge \bigg( \int_{\Delta^{n_1}} X^1_{s, i_1}\left(\frac {t_1} p\right) \cdots X^{n_1}_{s, i_1}\left(\frac {t_{n_1}} p\right) dt_1 \cdots dt_{n_1} \bigg)
\wedge  g_{i_1,i_2}
\\\cdots   g_{i_{p-1},i_p}
\wedge \bigg(
\int_{\Delta^{n_p}} X^1_{s, i_p}\left(\frac {p- 1 + t_1} p \right) \cdots X^{ n_p}_{s,i_p}\left(\frac {p- 1+t_{n_p}} p\right) dt_1 \cdots dt_{n_p} 
\bigg) ds \Bigg)
\end{multline}
where $g_{i_{k-1},i_k}$ is evaluated at $\gamma((k-1)/p)$, and the second sum is a sum over all $k$-element index sets $J\subset S$ of the sets
$S= \{(i_r,j): r=1,\dots, p,\text{ and } 1\leq j\leq n_r\}$, and singleton $(i_q,m) \in S - J$, and
\[
X^j_{s,i} = \left\{
\begin{array}{rl}
R_{s,i} & \text{if  } (i,j) \in J\\
A'_{s,i} & \text{if  } (i,j) = (i_q,m) \\
\iota A_{s,i} & \text{otherwise.}
\end{array} \right.
\]

Furthermore, we define the 
\emph{Bismut-Chern-Simons form}, associated to the choice $(p,\U)$, as
\[
BCS^{(p,\U)}(\nabla_s):=\sum_{k\geq 0} BCS^{(p,\U)}_{2k+1} \quad \in \Omega^{odd}(LM).
\]
\end{defn}

Heuristically, \eqref{eq:BCS} is much like formula \eqref{eq:BCh} for $BCh(\nabla_s)$, but with one copy of $A'_s$ shuffled throughout, and integrated over $s=0$ to $s=1$.

In appendix \ref{app:A} we show that $BCS_{2k+1}^{(p,\U)}$ is independent of subdivision integer $p$, and covering $\U$ of local trivializations of $E \to M$, and so it defines a global form $BCS_{2k+1}(\nabla_s)$ on $LM$. Hence, the total form 
\[
BCS(\nabla_s):=\sum_{k\geq 0} BCS_{2k+1}(\nabla_s) \quad \in \Omega^{odd}(LM)
\]
is also well defined. This form respects composition of paths of connections on $E\to M$, in the sense that for two paths of connections $\nabla_s$ and $\bar\nabla_s$ with $\nabla_1=\bar\nabla_0$, we have
\begin{equation}\label{EQU:BCS-composition}
 BCS(\bar\nabla_s\circ\nabla_s)=BCS(\bar\nabla_s)+BCS(\nabla_s),
\end{equation}
since the integral for $BCS(\nabla_s\circ\nabla_s)$ breaks into a sum of two integrals. It furthermore satisfies the following property.
  \begin{prop} \label{prop:restBCS}
 For any path $\nabla_s$ of connections on a complex vector bundle $E \to M$, the restriction of the Bismut-Chern-Simons form on $LM$ to $M$ equals the  Chern-Simons form,
\[
\rho^* BCS(\nabla_s) = CS(\nabla_s),
\]
where $\rho^* :\Om_{S^1} (LM) \to \Om(M)$ is the restriction to constant loops.
\end{prop}

\begin{proof} Consider the restriction of formula \eqref{eq:BCS}  to $M$, for any $p$ and $\U$ . Since the local forms $\iota A$ vanish on constant loops, the only non-zero integrands are those that contain only $R_s$ and $A'_s$. Now, $R_s$ is globally defined on $M$, as a form with values in $End(E)$, and $A'_s$ is a globally defined $1$-form on $M$, so we may take $p=1$ and $\U = \{M\}$ for the definition of $BCS_{2k+1}(\nabla_s)$. In this case, the formula for $BCS_{2k+1}^{(p,\U)}$ agrees with
the Chern-Simons form in \eqref{eq:CS} since $1/n!$ is the volume of the $n$-simplex.
\end{proof}

The fundamental homotopy formula relating the Bismut-Chern-Simons form and Bismut-Chern forms is the following.

\begin{thm} \label{thm:dBCS} Let $\nabla_s$ be a path of connections on $E \to M$. Then 
\[
(d+\iota)(BCS(\nabla_s))=BCh(\nabla_1)-BCh(\nabla_0).
\]
\end{thm}

\begin{proof} We'll first give the proof for the local expressions in \eqref{eq:BCS^U_{2k+1}} and \eqref{eq:BCh^U_{2k}}, and then 
explain how the same argument applies to the general global expressions \eqref{eq:BCS} and \eqref{eq:BCh}.
Let
\begin{multline*}
I_{2k+1}=
\sum_{n\geq k+1}\Bigg(
\sum_{1\leq j_1<\dots<j_k\leq n} \,\,
\sum_{\tiny\begin{matrix}r=1 \\ r\neq j_1,\dots, j_k \end{matrix}}^n
\\ \int_{\Delta^n}
 \iota A_s(t_1)\dots R_s(t_{j_1})  \dots A'_s(t_r)\dots R_s (t_{j_k})\dots \iota A_s(t_n)dt_1\dots dt_n\Bigg)
 \end{multline*}
be the integrand appearing in \eqref{eq:BCS^U_{2k+1}} so that 
\[
BCS^U_{2k+1}(\nabla_s)= Tr \int_0^1 I_{2k+1} ds 
\]
We first show that for each $s$ we have
\begin{multline} \label{d+[A_s,-]+i I_{2k+1}}
(d + \iota+[A_s(0),-]) \Bigg( \sum_{k \geq 0}  I_{2k+1} \Bigg) 
 =\sum_{k\geq 0}\Bigg(\sum_{n\geq k}
 \\
 \sum_{1\leq j_1<\dots<j_{k}\leq n}
  \int_{\Delta^n}\frac{\partial}{\partial s}\Bigg( \iota A_s (t_1)\dots R_s(t_{j_1})\dots R_s(t_{j_{k}})\dots \iota A_s(t_n) \Bigg) \,\, dt_1\dots dt_n  \Bigg)
 \end{multline}
The statement of the theorem (for the local case) will then follow from this by taking trace of both sides, integrating from $s=0$ to $s=1$, and using the fundamental theorem of calculus. Note also, that taking the bracket with $A_s(0)$ vanishes when taking the trace.

To prove \eqref{d+[A_s,-]+i I_{2k+1}} we evaluate $\frac{\partial}{\partial s}$ on the right-hand side of \eqref{d+[A_s,-]+i I_{2k+1}} as
\[
  \sum_{k\geq 0}   \,\,\sum_{n\geq k} \,\,
\sum_{1\leq j_1<\dots<j_{k}\leq n} \,\,
  \int_{\Delta^n}\Big(\omega+\eta\Big)\,\, dt_1\dots dt_n  
\]
where
\begin{eqnarray}\label{term-one}
\omega &=& \sum_{\tiny\begin{matrix}{ \ell=1}\\{\ell \neq j_1,\dots,j_{k}}\end{matrix}}^n 
 \iota A_s (t_1)\dots R_s(t_{j_1}) \dots \iota A'_s(t_\ell) \dots R_s(t_{j_{k}})\dots \iota A_s(t_n),
 \\ \label{term-two}
 \eta&=&\sum_{i=1}^k   \,\,
 \iota A_s (t_1)\dots R_s(t_{j_1})\dots R'_s(t_{j_i}) \dots R_s(t_{j_{k}})\dots \iota A_s(t_n).  
\end{eqnarray}
Thus, we need to show that the left-hand side of \eqref{d+[A_s,-]+i I_{2k+1}} consists of exactly the two kinds of terms given in \eqref{term-one} and \eqref{term-two}.

For the left-hand side of \eqref{d+[A_s,-]+i I_{2k+1}}, we first apply $\iota$ to $\sum_{k\geq 0} I_{2k+1}$. Since $\iota$ acts as a derivation and $\iota^2 = 0$, we have $\iota \iota A_s = 0$, so that we only obtain terms with exactly one $\iota R_s$ or $\iota A'_s$, \emph{i.e.} we get the following integrands (suppressing the variables $t_i$ for better readability):
\begin{eqnarray}
\label{iota-one}
&&  \pm \iota A_s \dots R_s \dots \iota R_s \dots A'_s \dots R_s \dots \iota A_s
\\
\label{iota-two}
&& \,\,\,\,\,\, \iota A_s \dots R_s \dots \dots \dots \iota A'_s \dots R_s \dots \iota A_s
\end{eqnarray}
In \eqref{iota-one} the factor $\iota R_s$ may appear anywhere in this product; in particular it may appear before the factor $A'_s$ or after that factor. Since $\iota A_s$ and $R_s$ are even, and $A'_s$ is odd, the sign ``$\pm$'' in \eqref{iota-one} is ``$+$'' if $\iota R_s$ appears before $A'_s$, and ``$-$'' if $\iota R_s$ appears after $A'_s$. Note, that \eqref{iota-two} is precisely the term \eqref{term-one} on the right-hand side of \eqref{d+[A_s,-]+i I_{2k+1}}.

Next, we apply the derivation $d$ to $\sum_{k\geq 0} I_{2k+1}$. We now obtain terms containing exactly one $d\iota A_s$, $d R_s$, or $dA'_s$, \emph{i.e.} (suppressing again the variables $t_i$):
\begin{eqnarray}
\label{d-one}
&&  \pm \iota A_s \dots R_s \dots d \iota A_s \dots A'_s \dots R_s \dots \iota A_s
\\
\label{d-two}
&&  \pm \iota A_s \dots R_s \dots d R_s \dots A'_s \dots R_s \dots \iota A_s
\\
\label{d-three}
&& \,\,\,\,\,\, \iota A_s \dots R_s \dots \dots \dots d A'_s \dots R_s \dots \iota A_s
\end{eqnarray}
Again, the sign is ``$+$'' if the $d$ term appears before $A'_s$, and ``$-$'' otherwise.
To evaluate \eqref{d-one}, we use the relation
\begin{equation}\label{diotaAs}
d(\iota A_s)= [d,\iota] A_s - \iota (d A_s)=\frac{\partial}{\partial t} A_s - \iota(d A_s).
\end{equation}
By the fundamental theorem of calculus, the integral over $\frac{\partial}{\partial t} A_s$ is given by evaluation at the endpoints of integration, \emph{i.e.} $\int_{t_{i-1}}^{t_{i+1}} \frac{\partial}{\partial t_i} A_s(t_i) \, dt_i=A_s(t_{i+1})-A_s(t_{i-1})$. Thus the variable $t_i$ has been removed, and either $A_s$ is being multiplied to its adjacent term on the right, or $(-A_s)$ is being multiplied to its adjacent term on the left. This can be further analyzed by considering the following four cases.
\begin{enumerate}
\item
 If $d\iota A_s$ is the first or last factor in a summand of $I_{2k+1}$, we obtain terms $-A_s(0)$ and $-A_s(1)$ from the evaluation at the endpoints. These two terms are precisely $-A_s(0) I_{2k+1}  - I_{2k+1} A_s(1)=-[A_s(0),I_{2k+1}]$ since $A_s(0) = A_s(1)$. Thus, this cancels with the bracket $[A_s(0),-]$ in \eqref{d+[A_s,-]+i I_{2k+1}}.
\item 
If $d\iota A_s$ is adjacent to $\iota A_s$, we obtain $ -\iota A_s A_s + A_s \iota A_s=- \iota(A_s \wedge A_s)$ which, when combined with $-\iota (d A_s)$ from \eqref{diotaAs} above, equals $-\iota(dA_s+A_s\wedge A_s)=-\iota R_s$. Each such term appearing in $d I_{2k+1}$ cancels with the corresponding term \eqref{iota-one} coming from $\iota I_{2k+3}$.
\item 
If $d\iota A_s$ is adjacent to $R_s$, we obtain terms $A_s R_s - R_s A_s$, which cancel with the corresponding term \eqref{d-two} in $d I_{2k+1}$ containing $d R_s$, since $dR_s + [A_s , R_s]= 0$ by the Bianchi identity.
\item
Finally, if $d\iota A_s$ is adjacent to $A'_s$, we get terms $A'_s A_s + A_s A'_s$ (both with a ``$+$'' sign, since $d$ has moved across the $1$-form $A'_s$). This combines with $dA'_s$ from \eqref{d-three} to give \eqref{term-two}, since $dA'_s+A'_s A_s + A_s A'_s=(dA_s+A_s\wedge A_s)'=R'_s$.
\end{enumerate}
Thus, we have shown identity \eqref{d+[A_s,-]+i I_{2k+1}}, and with this the claim of the theorem in the local case.

For the general case, using multi-linearity and a similar calculation shows that 
\[
(d+\iota)(BCS(\nabla_s))=BCh(\nabla_1)-BCh(\nabla_0).
\]
 The only new feature comes from the apparent terms $g_{ij}$ in \eqref{eq:BCS}, which are not in \eqref{eq:BCS^U_{2k+1}}.
For these, note that all the terms $ g_{ij} A_j$ and $A_i g_{ij}$ which appear from the fundamental theorem of calculus applied to $\frac{\partial}{\partial t} A_s$, cancel with  $d g_{ij}$, since  $ g_{ij} A_j - A_i g_{ij} = dg_{ij}$.
\end{proof}

\begin{cor}
For any two connections $\nabla_0$ and $\nabla_1$ on a complex vector bundle, the difference $Tr(hol(\nabla_1)) - Tr(hol(\nabla_0))$ of the trace of the holonomies is a function on $LM$ which is given by the contraction of a $1$-form on $LM$.
\end{cor}

\begin{proof}
For any path $\nabla_s$ from $\nabla_1$ to $\nabla_0$, the degree zero part of $(d+i)BCS(\nabla_s)$ is $i (BCS_1)(\nabla_s)$, which is the difference of the traces of the holonomies.
\end{proof}

\begin{cor} \label{Cor:BChcharacter}
For any complex vector bundles $E \to M$, there is a well defined Bismut-Chern class $[BCh(E)] = [BCh(E,\nabla)] \in H^{even}_{S^1}(LM)$, independent of the connection $\nabla$. 
\end{cor}

We remark that this corollary, and also Corollary \ref{cor:BChfromK} below, were first proven by Zamboni using completely different methods in \cite{Z}.

\begin{proof}
First, for any path of connections $\nabla_s$ from $\nabla_1$ to $\nabla_0$, $BCS(\nabla_s)$ is in the kernel of $d \iota + \iota d$ since 
$BCH$ is $(d+i)$-closed:
\[
(d \iota + \iota d) BCS( \nabla_s) = (d+\iota) (BCh(E,\nabla_1)-BCh(E,\nabla_0) ) = 0.
\]
The corollary now follows from Theorem \ref{thm:dBCS} since the space of connections is path connected.
\end{proof}

Let $K^0(M)$ be the even K-theory of complex vector bundles over $M$, i.e. the Grothendieck group associated to the 
semi-group of all complex vector bundles under direct sum. 
Elements in $K^0(M)$ are given by pairs $(E,E')$, thought of as the formal difference $E -E'$.
This is a ring under tensor product.
Using Corollary \ref{Cor:BChcharacter}, Proposition \ref{thm:BChsumtensor},  and Proposition \ref{prop:restBCh}, we have the following:

\begin{cor} \label{cor:BChfromK} There is a well defined ring homomorphism 
\[
[BCh]: K^0(M) \to H^{even}_{S^1}(LM)\
\] 
defined by $(E,\bar E) \mapsto BCh(E)-BCh(\bar E)$. Moreover, the following diagram commutes
\[
\xymatrix{
 &H^{\textrm{even}}_{S^1}(LM) \ar[d]^{\rho^*} \\
 K^0(M)  \ar [ru]^{[BCh]} \ar[r]^{[Ch]} & H^{even} (M) 
}
\]
where $[Ch]: K(M) \to H^{even}(M)$ is the ordinary Chern character to deRham cohomology,  and $\rho^*$ is the restriction to constant loops.
\end{cor}

\section{Further properties of the Bismut-Chern-Simons Form}

We now show that, up to $(d+\iota)$-exactness, $BCS(\nabla_s)$ depends only on the endpoints of the path $\nabla_s$.

\begin{prop}    \label{prop:bigon}
Let $\nabla_s^0$ and $\nabla_s^1$, for $0 \leq s \leq 1$ be two paths of connections on a complex vector bundle $E \to M$ 
with the same endpoints, i.e. $\nabla_0^0 = \nabla_0^1$ and $\nabla_1^0 = \nabla_1^1$.
Then 
\[
BCS(\nabla_s^1)-BCS(\nabla_s^0)\in \Omega^{odd}_{exact}(LM),
\]
 i.e. there is an even form $H \in \Omega^{even}_{S^1}(LM)$ such that 
 \[
 (d + \iota)H = BCS(\nabla_s^1)-BCS(\nabla_s^0).
 \]
\end{prop}
\begin{proof} Since the space $S$ of connections on $E$ is simply connected, there is a continuous function $F: [0,1] \times [0,1] \to S$ such that
$F(s,0) = \nabla_s^0$,  and $F(s,1) = \nabla_s^1$ for all $s \in [0,1]$, and $F(0,r) = \nabla_0^0 = \nabla_0^1$ and $F(1,r) = \nabla_1^0 = \nabla_1^1$ for all $r \in [0,1]$. We let $\nabla_s^r = F(s,r)$. The idea is to define an even form on $LM$ using 
the formula similar to that for $BCS(\nabla_s^r)$, expect with an additional term $\frac{\partial }{\partial r} \nabla_{s}^r$ shuffled in, and integrated from $r=0$ to $r=1$.
Explicitly, we let $H (\nabla_s^r)  = \sum_{k \geq  0} H_{2k+1} (\nabla_s^r)$ where $H_{2k+1} (\nabla_s^r)$ is given by 
\begin{multline} \label{Def-H_(2k+1)}
H_{2k+1} (\nabla_s^r) 
= Tr \Bigg ( \int_{r=0}^{r=1} \int_{s=0}^{s=1} \sum_{n_1,\dots,n_p\geq 0} \quad \sum_{\scriptsize
\begin{matrix}
J\subset S, |J| = k \\
(i_{q_1},m_1), (i_{q_2},m_2) \in S - J \\
(i_{q_1},m_1)\neq (i_{q_2},m_2)
\end{matrix}
} \\
 g_{i_p,i_1}
\wedge \bigg( \int_{\Delta^{n_1}} X^{r,1}_{s, i_1}\left(\frac {t_1} p\right) \cdots X^{r, n_1}_{s, i_1}\left(\frac {t_{n_1}} p\right) dt_1 \cdots dt_{n_1} \bigg)
\wedge  g_{i_1,i_2}
\\\cdots   g_{i_{p-1},i_p}
\wedge \bigg(
\int_{\Delta^{n_p}} X^{r,1}_{s, i_p}\left(\frac {p- 1 + t_1} p \right) \cdots X^{r, n_p}_{s,i_p}\left(\frac {p- 1+t_{n_p}} p\right) dt_1 \cdots dt_{n_p} 
\bigg) ds dr \Bigg)
\end{multline}
where $g_{i_{k-1},i_k}$ is evaluated at $\gamma((k-1)/p)$, and the second sum is a sum over all $k$-element index sets $J\subset S$ of the sets
$S= \{(i_\alpha,j): \alpha=1,\dots, p,\text{ and } 1\leq j\leq n_\alpha\}$, and distinct singletons $(i_{q_1},m_1), (i_{q_2},m_2) \in S - J$, and
\[
X^{r,j}_{s,i} = \left\{
\begin{array}{rl}
R_{s,i}^r & \text{if  } (i,j) \in J\\
\frac{\partial }{\partial s} A_{s,i}^r & \text{if  } (i,j) = (i_{q_1},m_1) \\
\frac{\partial }{\partial r}  A_{s,i}^{r} & \text{if  } (i,j) = (i_{q_2},m_2) \\
\iota A_{s,i}^r & \text{otherwise.}
\end{array} \right.
\]
Here $A_{s,i}^r$ is the local expression of $\nabla_s^r$ in $U_i$, with curvature $R_{s,i}^r$. It is shown in Proposition \ref{PROP:H-well-def} that $H(\nabla_r^s)$ is independent of the local trivialization chosen in the above expression for \eqref{Def-H_(2k+1)}, and thus defines a well defined global form on $LM$. 

Using the same techniques as in Theorem \ref{thm:dBCS} to calculate $(d+\iota)BCS(\nabla)$, and the equality of mixed partial derivatives, we can calculate that
\[
(d + \iota)  H (\nabla_s^r) = Z_1(\nabla_s^r) - Z_2(\nabla_s^r)
\]
where
\begin{multline*}
Z_1(\nabla_s^r)
= Tr \Bigg ( \int_{r=0}^{r=1} \int_{s=0}^{s=1} \sum_{n_1,\dots,n_p\geq 0} \quad \sum_{\scriptsize
\begin{matrix}
J\subset S, |J| = k \\
(i_{q_1},m_1) \in S - J
\end{matrix}
}
\\ \frac{\partial}{\partial r} \Bigg[  g_{i_p,i_1} 
\wedge \bigg(   \int_{\Delta^{n_1}} X^{r,1}_{s, i_1}\left(\frac {t_1} p\right) \cdots X^{r, n_1}_{s, i_1}\left(\frac {t_{n_1}} p\right) dt_1 \cdots dt_{n_1} \bigg)
\wedge  g_{i_1,i_2}
\\ \cdots   g_{i_{p-1},i_p}
\wedge \left(
\int_{\Delta^{n_p}} X^{r,1}_{s, i_p}\left(\frac {p- 1 + t_1} p \right) \cdots X^{r, n_p}_{s,i_p}\left(\frac {p- 1+t_{n_p}} p \right) \right) dt_1 \cdots dt_{n_p} \Bigg]  ds  dr\Bigg)
\end{multline*}
where $g_{i_{k-1},i_k}$ is evaluated at $\gamma((k-1)/p)$, and the second sum is a sum over all $k$-element index sets $J\subset S$ of the sets
$S= \{(i_\alpha,j): \alpha=1,\dots, p,\text{ and } 1\leq j\leq n_\alpha\}$, and singletons $(i_{q_1},m_1) \in S - J$, and
\[
X^{r,j}_{s,i} = \left\{
\begin{array}{rl}
R_{s,i}^r & \text{if  } (i,j) \in J\\
\frac{\partial }{\partial s} A_{s,i}^r & \text{if  } (i,j) = (i_{q_1},m_1) \\
\iota A_{s,i}^r & \text{otherwise.}
\end{array} \right.
\]
and 
\begin{multline*}
Z_2(\nabla_s^r)
= Tr \Bigg ( \int_{r=0}^{r=1} \int_{s=0}^{s=1} \sum_{n_1,\dots,n_p\geq 0} \quad \sum_{\scriptsize
\begin{matrix}
J\subset S, |J| = k \\
(i_{q_1},m_1) \in S - J
\end{matrix}
}
\\ \frac{\partial}{\partial s} \Bigg[  g_{i_p,i_1} 
\wedge \bigg(   \int_{\Delta^{n_1}} X^{r,1}_{s, i_1}\left(\frac {t_1} p\right) \cdots X^{r, n_1}_{s, i_1}\left(\frac {t_{n_1}} p\right) dt_1 \cdots dt_{n_1} \bigg)
\wedge  g_{i_1,i_2}
\\ \cdots   g_{i_{p-1},i_p}
\wedge \left(
\int_{\Delta^{n_p}} X^{r,1}_{s, i_p}\left(\frac {p- 1 + t_1} p \right) \cdots X^{r, n_p}_{s,i_p}\left(\frac {p- 1+t_{n_p}} p \right) \right) dt_1 \cdots dt_{n_p} \Bigg]  ds  dr\Bigg)
\end{multline*}
where $g_{i_{k-1},i_k}$ is evaluated at $\gamma((k-1)/p)$, and the second sum is a sum over all $k$-element index sets $J\subset S$ of the sets
$S= \{(i_\alpha,j): \alpha=1,\dots, p,\text{ and } 1\leq j\leq n_\alpha\}$, and singletons $(i_{q_1},m_1) \in S - J$, and
\[
X^{r,j}_{s,i} = \left\{
\begin{array}{rl}
R_{s,i}^r & \text{if  } (i,j) \in J\\
\frac{\partial }{\partial r}  A_{s,i}^{r} & \text{if  } (i,j) = (i_{q_2},m_2) \\
\iota A_{s,i}^r & \text{otherwise.}
\end{array} \right.
\]
Now, using the fundamental theorem of calculus with respect to $s$, we see that $Z_2(\nabla_s^r) =0$, because 
$\frac{\partial }{\partial r}  A_{0,i}^{r} = \frac{\partial }{\partial r}  A_{1,i}^{r}=0$, as $\nabla_0^r$ and $\nabla_1^r$ are constant.
On the other hand, using the fundamental theorem of calculus with respect to $r$ we have
\[
Z_1(\nabla_s^r) = BCS(\nabla_s^1) - BCS(\nabla_s^0)
\]
which shows $(d + \iota)H(\nabla_s^r) = BCS(\nabla_s^1) - BCS(\nabla_s^0)$ and completes the proof.
\end{proof}

\begin{defn}[BCS-equivalence] \label{defn:BCSexactness}
Let $E \to M$ be a complex vector bundle. We say two connections $\nabla_0$ and $\nabla_1$ on $E$ are $BCS$-equivalent 
if $BCS(\nabla_s)$  is $(d+\iota)$-exact for some path of connections $\nabla_s$ from $\nabla_0$ to $\nabla_1$.
\end{defn}

By Proposition \ref{prop:bigon}, if $BCS(\nabla_s)$ is $(d+\iota)$-exact for some path of connections $\nabla_s$ from $\nabla_0$ to $\nabla_1$, then $BCS(\nabla_s)$  is $(d+\iota)$-exact for any path of connections $\nabla_s$ from $\nabla_0$ to $\nabla_1$. Moreover, given two connections $\nabla_0$ and $\nabla_1$ on $E$, there is a well defined element 
\[
[BCS(\nabla_0,\nabla_1)]=[BCS(\nabla_s)] \in \Om^{odd}_{S^1}(LM) \Big/ Im(d + \iota),
\]
which is independent of the path $\nabla_s$ between $\nabla_0$ and $\nabla_1$. Two connections $\nabla_0$ and $\nabla_1$ are $BCS$-equivalent if and only if $[BCS(\nabla_0, \nabla_1)]=0$.

We remark that $BCS$-equivalence is an equivalence relation on the set of connections on a fixed bundle $E \to M$. Only transitivity needs checking, but it follows from the fact that 
\begin{equation} \label{BCSsum}
[BCS(\nabla_0,\nabla_2)] = [BCS(\nabla_0,\nabla_1)] + [BCS(\nabla_1,\nabla_2)] 
\end{equation}
since we may choose a path $\nabla_s$ from $\nabla_0$ to $\nabla_2$ that passes through $\nabla_1$, and then the 
integral over $s$ defining $BCS(\nabla_s)$ breaks into a sum.

The Bismut-Chern-Simons forms satisfy the following relations regarding direct sum and tensor product, which will be used to define loop differential K-theory.

\begin{thm} \label{prop:BCSsumtensor}
Let $E\to M$ and $\bar E \to M$ be two complex vector bundles, each with a path of connections $(E,\nabla_s) $ and $(\bar E,\bar \nabla_s)$ with $s \in [0,1]$, respectively. 
Let $\nabla_s \oplus \bar \nabla_s$ be the induced path of connections on $E \oplus \bar E$, and let 
$\nabla_r \otimes \bar \nabla_s:= \nabla_r \otimes  Id + Id \otimes \bar \nabla_s$ be the induced connections on $E \otimes \bar E$ for any $r,s\in [0,1]$. Then
\[
BCS(\nabla_s \oplus \bar \nabla_s) = BCS(\nabla_s) + BCS(\bar \nabla_s)
\]
and
\[
BCS(\nabla_0 \ot  \bar \nabla_s) = BCh(\nabla_0) \wedge BCS( \bar \nabla_s) \quad \quad 
BCS(\nabla_s \ot  \bar \nabla_1) = BCS(\nabla_s ) \wedge BCh( \bar  \nabla_1) 
\]
and so
\[
[BCS(\nabla_0 \ot  \bar  \nabla_0,\nabla_1 \ot  \bar  \nabla_1 )] = 
BCh(\nabla_0) \wedge [BCS( \bar  \nabla_0, \bar  \nabla_1)]+ [BCS(\nabla_0, \nabla_1)] \wedge BCh( \bar  \nabla_1) 
\]
\end{thm}

\begin{proof}
If $\nabla_s$ and $ \bar  \nabla_s$ are locally represented by $A_{s,i}$ and $B_{s,i}$, then $\nabla_s \oplus  \bar  \nabla_s$ is locally given
by the block matrixes with blocks  $A_{s,i}$ and $B_{s,i}$. Similarly this holds for transition functions, curvatures, and the derivatives $A'_{s,i}$ and $B'_{s,i}$. The result now follows from Definition \ref{defn:BCS}, since block matrices are a subalgebra, and trace and integral over $s$ are additive along blocks.

The proof that $BCS(\nabla_0 \ot  \bar  \nabla_s) = BCh(\nabla_0) \wedge BCS( \bar  \nabla_s)$ is almost identical to the calculation
in Theorem \ref{thm:BChsumtensor} that $BCh(\nabla \ot  \bar  \nabla) = BCh(\nabla) \wedge BCh( \bar  \nabla)$, using 
the additional fact that $\frac{\partial}{\partial s} (\nabla_0 \ot  \bar \nabla_s) = Id \ot \frac{\partial}{\partial s} ( \bar \nabla_s)$.
The claim $BCS(\nabla_s \ot  \bar  \nabla_1) = BCS(\nabla_s ) \wedge BCh( \bar  \nabla_1) $ is proved similarly.

For the last claim, we use \eqref{BCSsum} and the composition of paths of connections $\nabla_0\ot\bar\nabla_s$ (for $s\in[0,1]$) with $\nabla_s\ot\bar\nabla_1$ (for $s\in[0,1]$), to conclude
\begin{eqnarray*}
[BCS(\nabla_0 \ot  \bar \nabla_0,\nabla_1 \ot  \bar  \nabla_1 )]  &=& [BCS(\nabla_0 \ot  \bar  \nabla_s)] + [BCS(\nabla_s \ot  \bar  \nabla_1)] \\
&=& [BCh(\nabla_0) \wedge BCS( \bar  \nabla_s)]+ [BCS(\nabla_s) \wedge BCh( \bar  \nabla_1)] \\
&=& BCh(\nabla_0) \wedge [BCS( \bar  \nabla_s)]+ [BCS(\nabla_s)] \wedge BCh( \bar  \nabla_1) 
\end{eqnarray*}
where in the last step we have used that $(d + \iota)$ is a derivation of $\wedge$, and $BCh$ is $(d + \iota)$-closed.
\end{proof}

\begin{cor} [Cancellation law] \label{cor:cancel}
Let $E \to M$ be a complex vector bundle with a pair of connections $\nabla_0$ and $\nabla_1$, and let
and $(\bar E,\bar \nabla) \to M$ be a bundle with fixed connection.  
Then $[BCS(\nabla_0 \oplus \bar\nabla, \nabla_1 \oplus \bar\nabla)] = [BCS(\nabla_0, \nabla_1)]$.
\end{cor}

\begin{proof} By the previous theorem, for any path of connection $\nabla_s$ from $\nabla_0$ to $\nabla_1$, 
\[
BCS(\nabla_s \oplus \bar\nabla) = BCS(\nabla_s) + BCS(\bar\nabla) = BCS(\nabla_s).
\]
\end{proof}

\section{Gauge Equivalence, BCS equivalence, CS equivalence} \label{sec:implications}

In this section we clarify how the condition of $BCS$-equivalence, defined in the previous section, is related to the notions of gauge equivalence, and to Chern-Simons equivalence, the latter defined in \cite{SS}. The definitions we need are as follows.

\begin{defn} Let $E \to M$ be a complex vector bundle. Two connections $\nabla_0$ and $\nabla_1$ on $E$ are:
\begin{enumerate}
\item \emph{gauge equivalent} if there is a vector bundle automorphism $f: E \to E$ covering $id: M \to M$ such that $f^* \nabla_1 = \nabla_0$.
\item \emph{gauge-path equivalent} if there is a path 
$h_s: E \to E$ of vector bundle automorphisms covering $id: M \to M$ such that $h_0=id$ and $h_1^* \nabla_0 = \nabla_1$.
\item \emph{CS-equivalent} if there is a path of connections $\nabla_s$ from $\nabla_0$ to $\nabla_1$ such that $CS(\nabla_s)$ is $d$-exact. 
\end{enumerate}
\end{defn}
 
In general, gauge equivalence does not imply gauge path equivalence, but if the gauge group consisting of bundle automorphisms $f: E \to E$ covering $id: M \to M$ is path connected, then gauge-path equivalence and gauge equivalence coincide. It is shown in \cite{SS} that $CS$-equivalence is independent of path $\nabla_s$. This also follows from Propositions \ref{prop:bigon} and \ref{prop:restBCS}. 

All three of these are equivalence relations, and Figure \ref{fig} describes how these are related to $BCS$-equivalence. 
The entries in the diagram are each conditions on a pair of connections $\nabla_0$ and $\nabla_1$ on a fixed bundle. 
Note the four entries labeled $BCS$ or $CS$, is exact or closed, mean that $BCS(\nabla_s), CS(\nabla_s)$  
is exact or closed for some path of connections $\nabla_s$ from $\nabla_0$ to $\nabla_1$, and this is 
well defined independent of path $\nabla_s$, by Propositions \ref{prop:bigon} and \ref{prop:restBCS}.

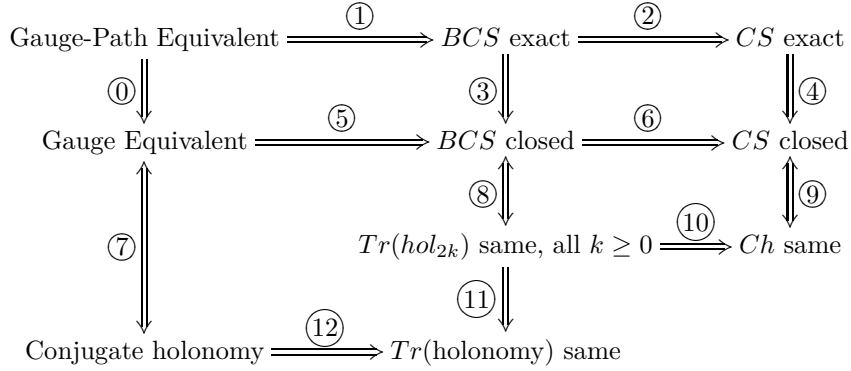
\begin{figure}
\[
\xymatrix{
\textrm{Gauge-Path Equivalent} \ar@{=>}[r]^-*+[o][F]{1} \ar@{=>}[d]_-*+[o][F]{0}  & \textrm{$BCS$ exact} \ar@{=>}[r]^*+[o][F]{2} \ar@{=>}[d]_*+[o][F]{3} & \textrm{$CS$ exact} \ar@{=>}[d]^*+[o][F]{4}  \\
  \textrm{Gauge Equivalent}  \ar@{<=>}[dd]_*+[o][F]{7} \ar@{=>}[r]^-*+[o][F]{5} & \textrm{$BCS$ closed} \ar@{=>}[r]^*+[o][F]{6} \ar@{<=>}[d]_*+[o][F]{8} & \textrm{$CS$ closed}  \ar@{<=>}[d]^*+[o][F]{9} \\
 & \textrm{$Tr(hol_{2k})$ same, all $k \geq 0$} \ar@{=>}[r]^-*+[o][F]{10} \ar@{=>}[d]_*+[o][F]{11} & \textrm{$Ch$ same} \\
 \textrm{Conjugate holonomy} \ar@{=>}[r]^*+[o][F]{12} & \textrm{$Tr($holonomy$)$ same}
}
\]
\caption{Diagram of implications for two connections on a bundle.}\label{fig}
\end{figure}

Clearly, gauge-path equivalence implies gauge equivalence, i.e. $\mycirc 0$ holds in Figure \ref{fig}. 
Implications $\mycirc{2}$ and $\mycirc 6$ follow from Proposition \ref{prop:restBCS} since the restriction map $\rho^* :(\Om^{S^1}(LM), d+ \iota)  \to (\Om(M), d)$  is a chain map sending $BCS(\nabla_s)$ to $CS(\nabla_s)$.
Similarly, implication $\mycirc {10}$ follows by Proposition \ref{prop:restBCh} since the restriction map 
$\rho^*$  sends $BCh(\nabla_i) = \sum_{k \geq 0} Tr(hol_{2k}(\nabla))$ to $Ch(\nabla)$.

Implications $\mycirc3$ and $\mycirc4$ follow since $(d+i)^2 BCS=0$ and $d^2 CS=0$, respectively,
while the  bi-conditional $\mycirc 7$ is standard from bundle theory: gauge equivalent connections have conjugate holonomy and the converse follows since if two connections have the same holonomy along all based loops (even at one point) then the connections are gauge equivalent (and the gauge equivalence can be constructed using the connections).

Implication $\mycirc5$ is Corollary \ref{cor:BCSclosedforGE}. Implications $\mycirc8$ and $\mycirc9$ follow from Theorem \ref{thm:dBCS} and Proposition \ref{prop:dCS}, respectively, while $\mycirc{11}$ follows since $hol_0$ is holonomy, by definition. 
Furthermore, $\mycirc{12}$ is true since trace is invariant under conjugation.

Only $\mycirc 1$ remains. We'll give only a sketch of the proof, since we won't need the result in what follows. If $\nabla_0$ and $\nabla_1$ are gauge-path equivalent then there is a path $h_s$ in the gauge group such that $h_0 = id$ and 
$h_s \nabla_0 = \nabla_s$ for all $s$. Then, on any local chart $U \subset M$ we can write $h^U_s: U \to G$, uniquely up to a choice of gauge, and have $(A^U_s)' = (d + [A^U_s, -]) ((h^U_s)^{-1} \frac{\partial}{\partial s} (h^U_s))$ for all $s$, where $A^U_s$ is the local expression of $\nabla_s$ on $U$.  

We define an even form $\omega_{h^U_s}$ on $LM$ by the same formula as for $BCS(\nabla_s)$ except we replace on each coordinate chart $U$ the $1$-form $A'_s$ by the function $(h^U_s)^{-1} \frac{\partial}{\partial s} (h^U_s)$. One then checks that $\omega_{h^U_s}$ determines a  well defined global even form $\omega_{h_s}$ on $LM$, independent of choices, using the same methods as in Appendix \ref{app:A} which show $BCS(\nabla_s)$ is well defined. Then we calculate that 
\[
(d+\iota) \omega_{h_s} = BCS(\nabla_s),
\]
by the same argument as used to compute $(d+\iota)BCS(\nabla_s)$, except at the new terms
$(h^U_s)^{-1} \frac{\partial}{\partial s} (h^U_s)$, where we use the relation $(d + [A^U_s, -]) ((h^U_s)^{-1} \frac{\partial}{\partial s} (h^U_s)) = (A^U_s)'$ and the facts that $h_s R_s = R_s h_s$ and $g_{ij} h_s^{U_i} = h^{U_j}_s g_{ij}$ on $U_i \cap U_j$, because $h_s$ is a gauge transformation.

\subsection{Counterexamples to converses} \label{sub:counterexs}

We give a single counterexample to the converses of implications $\mycirc 1$, $\mycirc 5$,  and $\mycirc{12}$ by constructing a bundle with a pair of connections that are $BCS$-equivalent, but do not have conjugate holonomy, as follows. 

Consider the trivial complex $2$-plane bundle  $\C^2 \times S^1 \to S^1$ over the circle. 
There is a path of flat connections given by 
\[
A_s = s 
\begin{bmatrix}
0 & dt \\
0 & 0 
\end{bmatrix},
\]
 so that $BCS(A_s)$ is a $1$-form on $LM$. Since
$A'_s = \begin{bmatrix}
0 & dt \\
0 & 0 
\end{bmatrix}$
and $A_s$ are upper triangular, those integrands in $BCS(A_s)$ containing $\iota A_s$ are zero, and we have
\[
BCS(A_s) = Tr \left( \int_0^1 A'_s ds \right) = 0.
\]
In particular $BCS(A_s)$ is exact. On the other hand, $A_0$ has holonomy along $S^1$ given by
$e^{\int A_0} = \begin{bmatrix}
1 & 0 \\
0 & 1
\end{bmatrix}$
while $A_1$ has holonomy along $S^1$ given by
$e^{\int A_1} = \begin{bmatrix}
1 & 1 \\
0 & 1
\end{bmatrix}
$, which are not conjugate. This shows the converse to both $\mycirc 1$, $\mycirc 5$ and $\mycirc{12}$ are false in general.

A counterexample to the converses of $\mycirc 2$ and $\mycirc 6$, and therefore also $\mycirc{10}$, is constructed as follows. Consider the trivial complex $2$-plane bundle $\C^2 \times S^1 \to S^1$ over the circle. 
For any $\alpha \in \R$ with $\alpha \neq 2k\pi$ for $k\in\Z$, consider the path of flat connections given by 
\[
A_s = s 
\begin{bmatrix}
0 & - \alpha dt \\
\alpha dt & 0 
\end{bmatrix},
\]
Then $CS(A_s)$ is a $1$-form on $M$, and is exact since
\[
CS(A_s) = Tr \left( \int_0^1 A'_s ds \right)  = 0.
\]
Similarly, $BCS(A_s)$ is a 1-form on $LM$, but $BCS(A_s)$ is not $(d+\iota)$-closed. Along the fundamental loop $\gamma$ of 
$S^1$ we have
\begin{eqnarray*}
(d + \iota) BCS(A_s)(\gamma) &=& BCh(A_1)(\gamma)  - BCh(A_0)(\gamma) \\
& =& Tr \left(
 \begin{bmatrix}
\cos \alpha & -\sin \alpha \\
\sin \alpha & \cos \alpha
\end{bmatrix}
-
 \begin{bmatrix}
1 & 0 \\
0  &1
\end{bmatrix}
\right)
\end{eqnarray*}
which is non-zero for $\alpha \neq 2k\pi$. This shows the converse to  $\mycirc 2$ and $\mycirc 6$ are both false.
We remark that since $BCS(A_s)$ is not closed, the endpoint connections are not gauge equivalent.

A counterexample to the converse of implication $\mycirc 3$ is given as follows.
Consider the trivial complex $2$-plane bundle $\C^2 \times S^1 \to S^1$ over the circle, with the path of flat connections given by 
\[
A_s = s 
\begin{bmatrix}
2 \pi i  \, dt &  dt \\
0 & 2 \pi i  \, dt 
\end{bmatrix},
\]
where $i=\sqrt{-1}$. Then $CS(A_s)$ is a non-exact $1$-form on $M$ since
\[
CS(A_s) = Tr \left( \int_0^1 A'_s ds \right)  = 4 \pi i \, dt.
\]
Therefore $BCS(A_s)$ is not exact. But we claim $BCS(A_s)$ is closed. To see this, note that for any loop $\gamma$ we have 
$(d + \iota) BCS(A_s)(\gamma) = BCh(A_1)(\gamma)  - BCh(A_0)(\gamma)$, and since curvature vanishes on the circle, the differential forms $BCh(A_i)(\gamma)$ are concentrated in degree zero and equal to the holonomy along $\gamma$. Since the holonomy depends only on the homotopy class of $\gamma$, it suffices to check that this expression vanishes on $\gamma^k$, for each $k \in \Z$, where 
$\gamma^k$ is $k$ times the fundamental loop in $S^1$. For this we have 
\[
BCh(A_1)(\gamma^k)  - BCh(A_0)(\gamma^k)  = Tr \left(
 \begin{bmatrix}
1 & 1 \\
0 & 1
\end{bmatrix}^k
-
 \begin{bmatrix}
1 & 0 \\
0  &1
\end{bmatrix}^k
\right)
 =0.
\]
Finally, we remark that the endpoint connections $A_1$ and $A_0$ are not gauge equivalent since the holonomies are not conjugate.

\section{Loop differential K-theory}
In this section we gather the previous results to define loop differential K-theory, and give some useful properties.
This definition given here is similar to the definition of (even) differential $K$-theory given in \cite{SS}, which uses $CS$-equivalence classes.

For any smooth manifold $M$ we can consider the collection of complex vector bundles $E \to M$ with connection $\nabla$.
Definition \ref{defn:BCSexactness} provides an equivalence relation on this set, $BCS$-equivalence, whose equivalence classes will be denoted by $\{(E,\nabla)\}$. We say $(E,\nabla)$ and $(\bar E, \bar \nabla)$ are isomorphic if there is a bundle 
isomorphism $\phi: \bar E \to E$ such that $\phi^*( \nabla) = \bar\nabla$. By Proposition \ref{prop:BCSindepchoices}, 
$\phi^*\{(E,\nabla)\} = \{\phi^*(E,\nabla)\}$, so we may consider the set of isomorphism classes of $BCS$-equivalence classes of bundles. 

By Theorem \ref{prop:BCSsumtensor} this set forms a commutative monoid $\mathcal{M}$ under direct sum, and the tensor product is well defined, commutative, and satisfies the distributive law. This assignment $M \mapsto \mathcal{M}(M)$ is contravariantly functorial in $M$.

\begin{defn} Let $M$ be a compact smooth manifold. Loop differential K-theory of $M$, denoted $L \widehat K^0(M)$, is the Grothendieck group of the commutative monoid $\mathcal{M}(M)$ of isomorphism classes of $BCS$-equivalences classes of finite rank complex vector bundles with connection over $M$.
This defines a contravariant functor from the category of smooth manifolds to the category of commutative 
rings.
\end{defn}

The Grothendieck functor $\mathcal{L}$ can be constructed by considering equivalence classes of pairs $(w,x) \in N \times N$, where
$(w,x) \cong (y,z)$ if and only if $w + z + k = y + x + k$ for some $k \in N$, and defining addition by $[w,x] + [y,z] = [w+y,x+z]$.
In this case, the identity element is represented by $(x,x)$ for any $x \in N$, and the monoid map $N \to \mathcal{L}N$ is given by $x \mapsto (x,0)$. A sufficient though not necessary condition that the map $N\to \mathcal L N$ is injective is that the monoid satisfies the cancellation law ($w+k=y+k \implies w=y$). For loop differential K-theory, the map $\mathcal{M}(M) \to L \widehat K(M)$ is injective since the monoid $\mathcal{M}(M)$ satisfies the cancellation law, by Corollary \ref{cor:cancel}.

\subsection{Relation to (differential) K-theory}

We have the following commutative diagram of ring homomorphisms
 \[
\xymatrix{
 & K^0(M) \ar[rd]^{[BCh]} & \\
 L \widehat K^0(M) \ar[ru]^f \ar [rd]^{BCh} & &H^{\textrm{even}}_{S^1}(LM) \\
 & \Omega^{\textrm{even}}_{(d+\iota)-cl} (LM) \ar[ru] &
}
\]
where $f$ is the map which forgets the equivalence class of connections, 
\[
BCh( \{(E,\nabla)\}, \{(\bar E, \bar \nabla)\} ) = BCh( \{(E,\nabla)\})- BCh( \{(\bar E, \bar \nabla)\})
\]
is well defined by Theorem \ref{thm:dBCS},
the map $[BCh]$ comes from Corollary \ref{cor:BChfromK}, and the map $\Omega^{\textrm{even}}_{(d+\iota)-cl} (LM) \to 
H^{\textrm{even}}_{S^1}(LM)$ is the natural map from the space of $(d+\iota)$-closed even forms given by the quotient by the image of $(d+\iota)$.

The analogous commutative diagram for differential K-theory was established in \cite{SS}, and in fact 
the commutative diagram above maps to this analogous square for differential K-theory, making the following commute.
\[
\xymatrix{
 &  K^0(M) \ar[rd]_{[BCh]}  \ar@{=}[rrd]^{id}  & & & \\
L \widehat K^0(M) \ar[ru]^f \ar [rd]_{BCh} \ar[rrd]^{\pi} & & H^{\textrm{even}}_{S^1}(LM) \ar[rrd]_{\rho^*}  &  K^0(M) \ar[rd]^{[Ch]}  & \\
& \Omega^{\textrm{even}}_{(d+\iota)-cl} (LM)  \ar[ru]  \ar[rrd]_{\rho^*}  & \widehat K^0(M)  \ar[ru]^g \ar [rd]^{Ch}  & & H^{\textrm{even}}(M)\\
& & & \Omega^{\textrm{even}}_{d-cl} (M)  \ar[ru] &
}
\]
Here $\rho^*$ is the restriction to constant loops, $\pi$ is well defined by Proposition \ref{prop:restBCS}, and $g$ is the forgetful map.

\begin{cor} \label{cor:LoopdKrefines}
The natural map $\pi: L \widehat K^0(M) \to \widehat K^0(M)$ from loop differential K-theory to differential $K$-theory is surjective, and in the case of $M = S^1$ has non-zero kernel. That is, the functor $M \mapsto L \widehat K^0(M)$ yields a refinement of differential $K$-theory.
\end{cor}

\begin{proof} Surjectivity follows from Proposition \ref{prop:restBCS}, since $L \widehat K^0(M)$ and $\widehat K^0(M)$ are defined from the same set of 
bundles with connection.

The second example in subsection \ref{sub:counterexs} provides two connections $\nabla$ and $ \bar \nabla$ on a bundle over $S^1$ such that $BCS(\nabla, \bar \nabla)$ is not $(d+\iota)$-closed, since these connections have different holonomy and thus are not $BCS$-equivalent or gauge equivalent. Nevertheless these connections are $CS$-equivalent, so that the induced element 
$\{E,\nabla\} - \{E, \bar \nabla \}$ in $L \widehat K^0(S^1)$ maps to zero in $\widehat K^0(S^1)$. Finally $(\{E,\nabla\}, \{E, \bar \nabla \})$ is nonzero in $L \widehat K^0(S^1)$, or equivalently $(\{E,\nabla\},0) \neq ( \{E, \bar \nabla \},0 )$, since $\{E,\nabla\} $ and $\{E,\bar \nabla\} $ have different trace of holonomy, and the map $\mathcal{M}(S^1) \to L \widehat K^0(S^1)$ is injective.
\end{proof}

On the other hand, we have the following. Let $G(M)$ denote the Grothendieck group of the monoid of complex vector bundles with connection over $M$, up to gauge equivalence, under direct sum. This is a ring under tensor product, and although this ring is often difficult to compute, we do have by Corollary \ref{cor:BCSgaugeindep} a well defined ring homomorphism 
$\kappa: G(M) \to L \widehat K^0(M)$.

\begin{cor} \label{cor:LoopdKcoarser}
The natural ring homomorphism $\kappa: G(M) \to L \widehat K^0(M)$ is surjective, and in the case of $M = S^1$ has non-zero kernel. That is, the functor $M \mapsto L \widehat K^)(M)$ is a strictly coarser invariant than $M\mapsto G(M)$, the Grothendieck group of all vector bundles with connection up to gauge equivalence.
\end{cor}

\begin{proof} Surjectivity follows again from the definition. The first example in subsection \ref{sub:counterexs} provides two connections $\nabla$ and $ \bar \nabla$ on a trivial bundle $E$ over $S^1$ that do not have conjugate holonomy, 
and so are not gauge equivalent, but are $BCS$-equivalent. Therefore, the induced element $((E,\nabla), (E, \bar \nabla ))$ in $G(S^1)$ maps to zero in $L \widehat K^0(S^1)$.  It remains to show that $((E,\nabla), (E, \bar \nabla ))$ is non-zero in $G(S^1)$,
i.e. that $( (E,\nabla ),0)$ and $( (E, \bar \nabla ),0)$ are not equal. 

This follows from a more general fact: if for some point $x \in M$, the holonomies of $\nabla$ and $\bar \nabla$ for loops based as $x$ are not related by conjugation by any automorphism of the fiber of $E$ over $x$, then $\nabla \oplus \tilde \nabla$ and 
$\bar \nabla  \oplus \tilde \nabla$ are not gauge equivalent for any $(\tilde E,\tilde \nabla)$. To see this, we verify the contrapositive. Suppose $\nabla \oplus \tilde \nabla$ and $\bar \nabla \oplus 
\tilde \nabla$ are gauge equivalent for some $(\tilde E, \tilde \nabla)$.
Then $\nabla \oplus \tilde \nabla$ and $\bar \nabla \oplus \tilde \nabla$ have conjugate holonomy for loops based at any point.  
But $hol(\nabla \oplus \tilde \nabla) = hol (\nabla) \oplus hol(\tilde \nabla)$ and similarly $hol(\bar \nabla \oplus \tilde \nabla) = 
hol(\bar \nabla) \oplus hol(\tilde \nabla)$. By appealing to the Jordan form, we see that $hol(\nabla)$ and $hol(\bar \nabla)$
are conjugate. 

The general fact implies the desired result, since for the given example, the holonomies of $(E,\nabla )$ and $(E, \bar \nabla )$ are not conjugate at any point of $x \in S^1$.
\end{proof}

\section{Calculating the ring $L \widehat K^0(S^1)$}\label{SEC:LK(S1)}

In this section we calculate the ring $L \widehat K^0(S^1)$ and show that the map $BCh: L \widehat K^0(S^1) \to 
\Omega^{\textrm{even}}_{(d+ \iota)-cl} (LS^1)$ is an isomorphism onto its image. It is instructive to calculate 
$L \widehat K^0(S^1)$ geometrically from the definition, and to independently calculate the image of the map from its definition,
observing that the map is an isomorphism. The techniques used for each case are somewhat different, and both may be useful for other 
examples.

Let us first calculate the image of the map $BCh: L \widehat K^0(S^1) \to \Omega^{\textrm{even}}_{(d+ \iota)-cl} (LS^1)$, which is contained in $\Omega^0_{(d+ \iota)-cl} (L S^1)$, since bundles over the circle are flat. The space $LS^1$ has countably many components $L^k S^1$, where $L^k S^1$ contains the $k^{th}$ power of the fundamental loop $\gamma$ of $S^1$ at some fixed basepoint. An element of $\Omega^0_{(d+ \iota)-cl} (L S^1)$ is uniquely determined 
by its (constant) values on each $L^k S^1$. Notice that $BCh(\gamma^k) = Tr(hol(\gamma^k))= Tr(hol(\gamma)^k)$,
and so if $hol(\gamma)$ has eigenvalues $\lambda_1,  \ldots ,\lambda_n$, then 
$BCh \big|_{L^k S^1} =  \lambda_1^k +  \ldots +\lambda_n^k$. It is a fact that if invertible matrices $A$ and $B$ satisfy $Tr(A^k) = Tr(B^k)$ for all $k \in \N$ then $A$ and $B$ have the same eigenvalues. Therefore, the map $BCh$ can be lifted to the map $\pi$:
\[
\xymatrix{
 & \coprod_{n \in \N} (\C^*)^n/ \Sigma_n \ar[d]^{i}  \\
\mathcal{M}(S^1) \ar[r]^-{BCh} \ar[ur]^{\pi} &   \Omega^{0}_{(d+\iota)-cl} (L S^1)   
}
\]
where the map $i$ is given by setting $i([ \lambda_1,  \ldots ,\lambda_n])$ to be $\lambda_1^k +  \ldots +\lambda_n^k$ on $L^k S^1$.

The set $  \coprod_{n \in \N} (\C^*)^n/ \Sigma_n $ is a monoid under concatentation, and there is a 
commutative product, given by $[ \lambda_1,  \ldots ,\lambda_n] * [\rho_1, \ldots , \rho_m] = [ \lambda_1 \rho_1,  \ldots ,
 \lambda_i \rho_j, \ldots ,  \lambda_n \rho_m]$, which satisfies the distributive law. The map $i$ is a homomorphism with 
 respect to these structures, and an inclusion. Moreover, $BCh$ maps onto the image of $i$, since we can construct a bundle over $S^1$ of any desired holonomy, and therefore any desired eigenvalues.

The Grothendieck functor $\mathcal{L}$ takes surjections to surjections, and injections to injections if the target monoid satisifies the cancellation law. In particular, groups are monoids satifying the cancellation law, and the Grothendieck functor is the identity on groups. Therefore we can apply the Grothendieck functor,  and obtain the following commutative diagram of rings
\[
\xymatrix{
 & \mathcal{L} \left( \coprod_{n \in \N} (\C^*)^n/ \Sigma_n \ar[d]^{i} \right)  \\
L \widehat K(S^1) \ar[r]^-{BCh} \ar[ur]^{\pi} &   \Omega^{0}_{(d+\iota)-cl} (L S^1)   
}
\]
This shows that $L \widehat K^0(S^1)$ maps surjectively onto the ring $\mathcal{L} \left( \coprod_{n \in \N} (\C^*)^n/ \Sigma_n \right)$, 
which imbeds into $\Omega^{0}_{(d+\iota)-cl} (L S^1) $. Since the diagram above commutes, this calculates the image of $BCh$.

We now calculate the ring $L \widehat K^0 (S^1)$ directly from the definition. We need the following

\begin{lem} Every $\C^n$-bundle with connection $(E \to S^1, \nabla)$ over $S^1$ is isomorphic,
as a bundle with connection, to one of the form $(\C^n \times S^1 \to S^1$, $\nabla = Jdt)$,
where J is a constant matrix in Jordan form.
\end{lem}

\begin{proof} A $\C^n$-bundle with connection over $S^1$ is uniquely determined
up to isomorphism by its holonomy along the fundamental loop, which is a well defined element $[g] \in GL(n,\C)/ \sim$, 
where the latter denotes conjugacy classes of $GL(n,\C)$.

The exponential map from all complex matrices $M(n,\C)$ respects conjugacy
classes, it is surjective, so that it is surjective on conjugacy classes, and every conjugacy class is represented by a Jordan form.

Given a bundle with connection over $S^1$, let $[g]$ be the conjugacy class that determines it up to isomorphism. 
We can choose $J$ in Jordan form so that $[e^J] = [g] \in GL(n,\C)/ \sim$. Regard $Jdt$ as a connection on the trivial
bundle over $S^1$. Since the connection is constant, $e^{J}$ is the holonomy of this connection
along the fundamental loop, which completes the proof.
\end{proof}

By the Lemma, an element in the monoid $\mathcal{M}(S^1)$ which defines $L \widehat K^0(S^1)$ can always be represented by a bundle with connection of the form $(\C^n \times S^1 \to S^1$, $\nabla = Jdt)$.  The next lemma gives a sufficient condition for when two such are $BCS$-equivalent.

\begin{lem} Let $A_0dt$ and $A_1dt$ be constant connections on the trivial $\C^n$-bundle over $S^1$ where $A_0$ and $A_1$ are in Jordan form. If $A_0$ and $A_1$ have the same diagonal entries, then these connections are $BCS$-equivalent.
\end{lem}

\begin{proof} There is a path $A_s dt$ from $A_0dt$ to $A_1dt$ which is constant on the diagonal, a function of $s$ on the super-diagonal, and zero in all other enties. Therefore the non-zero entries of $A'_s dt$ are all on the super-diagonal.
Then the integrand defining $BCS(A_s dt)$ is zero on the diagonal, so $BCS(A_s dt)=0$, and therefore the connections are $BCS$-equivalent.
\end{proof}

Since for each Jordan form there is a diagonal matrix with the same diagonal entries we have

\begin{cor} \label{cor:surj} 
Every element in the monoid $\mathcal{M}(S^1)$ is represented by a bundle with
connection over $S^1$ in the following form:
$(\C^n \times S^1 \to S^1, Adt)$ where $A$ is constant diagonal matrix.
\end{cor}

It remains to determine when two representatives of the form $(\C^n \times S^1, Adt)$, where $A$ is constant diagonal matrix,
determine the same element in the monoid $\mathcal{M}(S^1)$. We first need

\begin{lem}
Let $A$ and $B$ be any two connections on the trivial bundle $\C^n \times S^1 \to S^1$. 
If $A$ and $B$ are isomorphic, or  $A$ and $B$ are BCS-equivalent, then $hol_\gamma(A)$ and $hol_\gamma(B)$ have the same eigenvalues, where $\gamma$ is the fundamental loop.
\end{lem}

\begin{proof}
If $A$ and $B$ are isomorphic then their holonomy are conjugate, so they have the same eigenvalues.

Secondly, if $A$ and $B$ are $BCS$-equivalent then $BCS(\nabla_s)$ is exact for some path $\nabla_s$ from $A$ to $B$. So, $BCS(\nabla_s)$ is $(d + \iota)$-closed, so that $BCh(A) = BCh(B)$. Therefore, for the $k^{\text{th}}$ power of the fundamental loop we have $Tr(hol_\gamma(A)^k) = Tr(hol_\gamma(B)^k)$, for all $k$, so that $hol_\gamma(A)$ and $hol_\gamma(B)$ have the same eigenvalues. 
\end{proof}

\begin{prop} \label{prop:inj}
If $Adt$ and $Bdt$ are constant connections on $\C^n \times S^1 \to S^1$, where $A$ and $B$ are diagonal, such that $\{(\C^n \times S^1,Adt)\}$ and $\{(\C^n \times S^1, Bdt )\}$ define the same class in $\mathcal M(S^1)$, then:
\begin{enumerate}
\item $hol_\gamma(Adt)$ and $hol_\gamma(Bdt)$ have the same eigenvalues, where $\gamma$ is the fundamental loop.
\item $Adt$ and $Bdt$ are isomorphic.
\item After possible re-ordering, the eigenvalues of $Adt$ and $Bdt$ each differ by some integer multiple of $2\pi i$.
\end{enumerate}
\end{prop}

\begin{proof} 
If $Adt$ and $Bdt$ are as given and represent the same class in $\mathcal M(S^1)$, then there exists a connection $C$ such that $Adt$ and $C$ are $BCS$-equivalent, and $Bdt$ and $C$ are isomorphic. By the previous Lemma, the holonomy along $\gamma$ for all three connections $Adt$, $Bdt$, and $C$ have the same eigenvalues. This proves $1)$, and therefore $2)$, since
connections on a bundle over the circle are determined up to equivalence by the conjugacy class of holonomy, and $A$ and $B$ are diagonal.  Finally, $3)$ follows since the eigenvalues of $hol_\gamma(Adt) = e^A$ are the complex exponential of the eigenvalues of $A$.
\end{proof}

\begin{cor}
There is a bijection $\coprod_{n} \left(\C/ \Z\right)^n  / \Sigma_n \to \mathcal{M}(S^1)$ where
 $\Z$ is the subgroup of $\C$ given by $\{ 2 k \pi i \}$ for $k$ an integer, and the symmetric group $\Sigma_n$ acts on
 $\left(\C/ \Z\right)^n$ by reordering. This bijection is a semi-ring homomorphism with respect to the operations on 
 $\coprod_{n} \left(\C/ \Z\right)^n  / \Sigma_n$ given by concatentation, and 
\[
(a_1,  \ldots , a_n) * (b_1, \ldots , b_m) = (a_1+ b_1,  \ldots , a_i +b_j, \ldots ,  a_n+ b_m)
 \]
\end{cor}

\begin{proof} The map sends $(a_1,  \ldots , a_n)$ to the equivalence class of the trivial bundle with connection given by the constant diagonal matrix with entries $a_1, \ldots , a_n$. It is well defined since reordering the $a_i$,
or changing some $a_i$ by an element in $\Z$ produces an isomorphic bundle with connection. It is straightforward to check it 
is a semi-ring homomorphism. It is surjective by Corollary \ref{cor:surj} above, and injective by the Proposition \ref{prop:inj} above.
\end{proof}

More intuitively, the elements of $\mathcal{M}(S^1)$ are determined uniquely by log of the spectrum of holonomy.
It follows that the group $L \widehat K^0 (S^1)$ is simply the Grothendieck group 
 $\mathcal{L} \left( \coprod_{n} \left(\C/ \Z\right)^n / \Sigma_n \right)$ of this monoid 
 $\coprod_{n} \left(\C/ \Z\right)^n / \Sigma_n$. Finally we have:

\begin{prop}
The map $BCh: L \widehat K^0 (S^1) \to \Omega^{0}_{(d+\iota)-cl} (L S^1)$ is an isomorphism onto its image.
\end{prop} 

\begin{proof}
Via the isomorphisms 
\[
L \widehat K^0 (S^1) \cong \mathcal{L}\left( \coprod_{n} \left(\C/ \Z\right)^n / \Sigma_n \right) 
\quad \textrm{and} \quad 
Im(BCh) \cong \mathcal{L} \left( \coprod_{n \in \N} (\C^*)^n/ \Sigma_n \right)
\]
the map is induced from the map on monoids given 
by $(a_1,  \ldots , a_n) \mapsto (e^{a_1},  \ldots , e^{a_n})$.
\end{proof}

We can also calculate the following maps:
\[
G(S^1) \to L \widehat K^0 (S^1) \to \widehat K^0 (S^1).
\]
The Grothendieck group $G(S^1)$ of all bundles with connection over $S^1$, up to isomorphism, is isomorphic
to the Grothendieck group  of the monoid of conjugacy classes in $GL(n,\C)$ under block sum and tensor product. The isomorphism is given by holonomy. 
The group $\widehat K^0 (S^1)$ is isomorphic to $\Z \oplus \C/ \Z$, as can be computed by the character diagram in \cite{SS}, or directly using a variation of the argument above used to calculate $L \widehat K^0 (S^1)$. With respect to these isomorphisms,
the first map $G(S^1) \to L \widehat K^0 (S^1)$ is given by taking $\log$ of the eigenvalues of a conjugacy class, while the second map $L \widehat K^0 (S^1) \to \widehat K^0 (S^1)$ is induced by $ (a_1,  \ldots , a_n) \mapsto (n,a_1 + \dots + a_n)$, where
 the sum is reduced modulo $\Z = \{ 2 k \pi i \}$.

\appendix

\section{} \label{app:A}

In this appendix we prove that the Bismut-Chern-Simons form $BCS(\nabla_s)$, associated to a path of connections
$\nabla_s$ on a vector bundle $E\to M$, is a well defined global differential form on $LM$. We also gather some useful corollaries.

Let  $ \{U_i\}$ be a covering of $M$ over which the bundle is locally trivialized, and write the connections $\nabla_s$ locally as $A_{s,i}$ on $U_i$, with curvature $R_{s,i}$.
For any $p\in \mathbb N$, and $p$ open sets $\U=(U_{i_1},\dots, U_{i_p})$ from the cover $\{U_i\}$, there is an induced open subset $\NN(p,\U)\subset LM$ given by
\[
\NN(p,\U)=\left\{\gamma\in LM: \left(\gamma\Big|_{\big[\frac{k-1}{p},\frac{k}{p}\big]}\right)\subset U_{i_j}, \forall j=1,\dots,p \right\}.
\]
Note that the collection $\{\NN(p,\U)\}_{p,i_1,\dots,i_p}$ forms an open cover of $LM$.

For a given loop $\gamma \in LM$ we can choose sets $U_1, \ldots , U_p$ that cover a subdivision of $\gamma$ into 
$p$ subintervals $[(k-1)/p , k/p]$, and for this choice we define
\begin{multline*}
BCS^{(p,\U)}_{2k+1} (\nabla_s)
= Tr \Bigg ( \int_0^1 \sum_{n_1,\dots,n_p\geq 0} \quad \sum_{\scriptsize
\begin{matrix}
J\subset S, |J| = k \\
(i_q,m) \in S - J
\end{matrix}
}\quad g_{i_p,i_1}(\gamma (0)) 
\\
\wedge \bigg( \int_{\Delta^{n_1}} X^1_{s, i_1}\left(\frac {t_1} p\right) \cdots X^{n_1}_{s, i_1}\left(\frac {t_{n_1}} p\right) dt_1 \cdots dt_{n_1} \bigg)
\wedge  g_{i_1,i_2}(\gamma( \scalebox{0.8}{$\frac 1 p$}))\cdots   g_{i_{p-1},i_p}(\gamma(\scalebox{0.8}{$ \frac{p-1}p$}))
\\
\wedge \bigg(
\int_{\Delta^{n_p}} X^1_{s, i_p}\left(\frac {p- 1 + t_1} p \right) \cdots X^{ n_p}_{s,i_p}\left(\frac {p- 1+t_{n_p}} p\right) dt_1 \cdots dt_{n_p} 
\bigg) ds \Bigg)
\end{multline*}
where 
 the $j^{th}$ integral  is over ${\Delta^{n_j}} :=  \{ (j-1)/p \leq t_1\leq \dots \leq t_j \leq j/p \} $. The second sum is a sum over all $k$-element index sets $J\subset S$ of the sets
$S= \{(i_r,j): r=1,\dots, p,\text{ and } 1\leq j\leq n_r\}$, and singleton $(i_q,m) \in S - J$,
\[
X^j_{s,i} = \left\{
\begin{array}{rl}
R_{s,i} & \text{if  } (i,j) \in J\\
A'_{s,i} & \text{if  } (i,j) = (i_q,m) \\
\iota A_i & \text{otherwise.}
\end{array} \right.
\] 
Recall the
\emph{Bismut-Chern-Simons form} associated to the choice $(p,\U)$ is
\[
BCS^{(p,\U)}(\nabla_s):=\sum_{k\geq 0} BCS^{(p,\U)}_{2k+1}(\nabla_s) \quad \in \Omega^{odd}(LM).
\]

We will need the following lemma, which will be used to show that  $BCS^{(p,\U)}_{2k+1} $ is invariant under subdivision in the sense that, if we increase $p$ and repeat coordinate neighborhoods in $\U$, then the expression does not change. In fact the property follows from considering the integrand $I^{(p,\U)}_{2k+1}$ of $BCS^{(p,\U)}_{2k+1} $ for fixed $s$ given by
\begin{multline*}
I^{(p,\U)}_{2k+1} = \sum_{n_1,\dots,n_p\geq 0} \quad \sum_{\scriptsize
\begin{matrix}
J\subset S, |J| = k \\
(i_q,m) \in S - J
\end{matrix}
}\quad g_{i_p,i_1} 
\\
\wedge \bigg( \int_{\Delta^{n_1}} X^1_{i_1} \left(\frac {t_1} p\right) \cdots X^{n_1}_{i_1}\left(\frac {t_{n_1}} p\right) dt_1 \cdots dt_{n_1} \bigg)
\wedge  g_{i_1,i_2}
\\ \cdots   g_{i_{p-1},i_p}
\wedge \bigg(
\int_{\Delta^{n_p}} X^1_{i_p}\left(\frac {p- 1 + t_1} p \right) \cdots X^{ n_p}_{i_p} \left(\frac {p- 1+t_{n_p}} p\right) dt_1 \cdots dt_{n_p} 
\bigg) 
\end{multline*}
 interpreted as above with
\[
X^j = \left\{
\begin{array}{rl}
R & \text{if  } (i,j) \in J\\
A' & \text{if  } (i,j) = (i_q,m) \\
\iota A & \text{otherwise.}
\end{array} \right.
\] 
with $A = A_s$, $A' = A'_s$ and $R = R_s$, where we have dropped the dependence on $s$. 

\begin{lem} \label{lem:BCSdiv} Let $k \geq 0$ and 
let $A, A'$, and $R$ be forms on $U \subset M$ with values in $\gl$. Let $\gamma_1 :[0,1] \to U$ and 
$\gamma_2: [0,1] \to U$ such that $\gamma_1(1) = \gamma_2(0)$, and let $\gamma=\gamma_2 \circ  \gamma_1: [0,1] \to U$ be the composition of the paths. Then 
\[
I^{(1,\{U\})}_{2k+1} (\gamma) = I^{(2,\{U,U\})}_{2k+1}(\gamma_2 \circ  \gamma_1).
\]

Similarly, for any $p \geq 0$ and collection $\U=(U_{i_1},\dots, U_{i_p})$, subdivide each of the $p$ intervals of $[0,1]$ into its $r$ subintervals, and let $\U'$ be the cover using the same open set $U_{i_j}$ for all of the $r$ subintervals of the $j^{\text{th}}$ interval, 
\[ U'_{i_{m\cdot r-r+1}}= \dots =U'_{i_{m\cdot r}} =U_{i_m}, \]
for $1 \leq m \leq p$. Then $\NN(p,\U)=\NN(r\cdot p, \U')$, and for $\gamma\in \NN(p,\U)$ and
\[
I^{(p, \U )}_{2k+1} (\gamma) = I^{(r \cdot p,\U ' )}_{2k+1}(\gamma).
\]
\end{lem}

\begin{proof} 
We denote by $\Delta^j_{[r,s]} = \{ r \leq t_1\leq \dots \leq t_j \leq s \} $. For the first statement we must show
\begin{multline} \label{lem:goal}
 \sum_{\ell \geq 0} \quad \sum_{\scriptsize
\begin{matrix}
J\subset S_\ell, |J| = k \\
m \in S_\ell - J
\end{matrix}
}
\bigg( \int_{\Delta^{\ell}_{[a,c]}} X^1 \left (t_1\right) \cdots X^{\ell}\left(t_{\ell}  \right) dt_1 \cdots dt_{\ell} \bigg)  \\
= 
 \sum_{n, m \geq 0} \quad \sum_{\scriptsize
\begin{matrix}
L \subset T_{n,m}, |L| = k \\
q \in T_{n,m} - L
\end{matrix}
}\quad 
\bigg( \int_{\Delta^{n}_{[a,b]}} Y^1_1 \left( {t_1}\right) \cdots Y^{n}_1\left( {t_{n}} \right) dt_1 \cdots dt_{n} \bigg)
\\ \wedge \bigg(
\int_{\Delta^{m}_{[b,c]}} Y^1_2\left( t_1 \right) \cdots Y^{ m}_2 \left(t_{m}\right) dt_1 \cdots dt_{m} 
\bigg) 
\end{multline}
where $S_\ell = \{1,2 \dots, \ell \}$,
\[
X^j = \left\{
\begin{array}{rl}
R & \text{if  } j \in J\\
A' & \text{if  } j=m \\
\iota A & \text{otherwise,}
\end{array} \right.
\] 
and $T_{n,m} = \{ (i,j_i) | \, i=1,2, \text{ and } \, \,  1\leq j_1 \leq n , 1 \leq j_2 \leq m\}$, and 
\[
Y^j_i = \left\{
\begin{array}{rl}
R & \text{if  } (i,j) \in L\\
A' & \text{if  } (i,j) = q \\
\iota A & \text{otherwise.}
\end{array} \right.
\] 
Note, that we are using the fact that the transition function $g=id$ on $U \cap U$ for the right hand side of \eqref{lem:goal}.
The proof becomes apparent using the calculus notation for the integral over $\Delta^k_{[r,s]}$,
 \[
\int_{\Delta^k_{[r,s]}}(\dots) dt_1\dots dt_k= \int_{r}^{s}  \int_{r}^{t_k} \cdots  \int_{r}^{t_3} \int_{r}^{t_2} (\dots) dt_1\dots dt_k.
 \]
 We 
  repeatedly use the relation $\int_a^b + \int_b^{t_{j}} = \int_a^{t_{j}}$  to add all the terms on the right hand side of \eqref{lem:goal}, over all  $n$ and $m$ such that $n+m = \ell$, where $\ell$ is fixed, giving each summand of the left hand side of \eqref{lem:goal}. 
 
 The second statement is proved similarly using the fact that, for $1 \leq m \leq p$ and $1 \leq s \leq r-1$,  we have $U_{i_{m\cdot r-s}} = U_{i_{m\cdot r-s+1}}$ and $g = g_{i_{m\cdot r-s} , i_{m\cdot r-s+1}} = id$.
\end{proof}

We are now ready to prove

\begin{prop} \label{prop:BCSindepchoices}
The locally defined $BCS^{(\U,p)}(\nabla_s)$ determine a well defined global form on $LM$ independent of the choice of local trivialization charts $\U$, and 
the subdivision integer $p \in \N$. 
\end{prop}

\begin{proof} It suffices to show that each odd form $BCS^{(\U,p)}_{2k+1}(\nabla_s)$ is a well defined global form on $LM$, independent of the choice of local trivialization charts $\U$ and the subdivision integer $p \in \N$.

We'll prove the following two properties:
\begin{enumerate}
\item Subdivision: Fix $p$ and $\U = \{ U_{i_1},\dots, U_{i_p} \}$. Subdivide each of the $p$ intervals of $[0,1]$ into $r$ subintervals, and use the same open set $U_{i_j}$ for all of the $r$ subintervals of the $j^{\text{th}}$ interval, to give a new cover $\U'$ with 
\[
U'_{i_{m\cdot r-r+1}}= \dots =U'_{i_{m\cdot r}} =U_{i_m}
\]
 for $1 \leq m \leq p$. 
Then $\NN(p,\U)=\NN(r\cdot p, \U')$, and $BCS^{(\U,p)}_{2k+1}(\nabla_s) = BCS^{(\U',r \cdot p)}_{2k+1}(\nabla_s)$.
\item Overlap: Assume that $p\in \N$, and that $\U = \{U_{i_1} , \ldots , U_{i_p} \}$, and 
 $\U' = \{U_{j_1} , \ldots , U_{j_p} \}$. Denote by $\U \cap \U' = \{ (U_{i_1} \cap U_{j_1}), \ldots , (U_{i_p} \cap U_{j_p}) \}$, and assume furthermore, that $\gamma \in  \NN(p,\U \cap \U')$. Then, we have 
\[
BCS^{(\U',p)}_{2k+1}(\nabla_s) (\gamma) = BCS^{(\U,p)}_{2k+1}(\nabla_s) (\gamma).
\]
\end{enumerate}
The proposition follows from these two facts, since for  $\gamma\in \NN(p,U_{i_1},\dots)\cap \NN(p',U_{j_1},\dots)$, we may assume by (1) that $p=p'$, and then by $(2)$ that the forms agree on the overlap.

Note that (1) follows from Lemma \ref{lem:BCSdiv}, which is the analogous statement for the integrand $I^{(p,\U)}_{2k+1} $.

We now prove (2). Since trace is invariant under conjugation, it suffices to show that for each fixed $s$ the integrand 
$I^{(p,\U)}_{2k+1} $ changes by conjugation if we perform a collection of local gauge transformations on each $U_i \in \U$.  In fact, it suffices to prove this for the integral expression $\int_{\Delta^{n_j}}$ on the $j^{th}$ subinterval,
since the sum defining $BCS^{(\U,p)}_{2k+1}(\nabla_s)$ can be re-ordered as a sum first over all $1 \leq j \leq p $ and $n_j \geq 0$, where the forms $A'_s$  and $R_s$ vary on this interval for each possible arrangement on the remaining intervals.
To this end, we'll drop the $s$ dependence and it suffices to show for each $k \geq 0$ that 
\begin{multline*} 
g(0) \Bigg( \sum_{n\geq k+1}\sum_{\tiny\begin{matrix}{1\leq i_1,<\dots<i_k\leq n}\\{1\leq r\leq n, \forall j: r\neq i_j}\end{matrix}}  \int_{\Delta^n_{[0,1/p]}}
 \\
 \iota A_j (t_1)\dots R_j(t_{i_1})
 \dots A'_j(t_r)\dots R_j  (t_{i_k})\dots \iota A_j (t_n) \, dt_1\dots dt_n \Bigg) 
 \end{multline*}
\begin{multline*} 
 =
 \Bigg( \sum_{n\geq k+1}\sum_{\tiny\begin{matrix}{1\leq i_1,<\dots<i_k\leq n}\\{1\leq r\leq n, \forall j: r\neq i_j}\end{matrix}}  \int_{\Delta^n_{[0,1/p]}} 
 \\
 \iota A_i (t_1)\dots R_i(t_{i_1})
 \dots A'_i(t_r)\dots R_i  (t_{i_k})\dots \iota A_i (t_n) \, dt_1\dots dt_n \Bigg) g\left(\frac 1 p\right)
 \end{multline*}
 where $\Delta^n_{[0,1/p]} = \{0 \leq t_1 \leq \dots \leq t_k \leq 1/p\}$, $g = g_{i,j}: U_i \cap U_j \to Gl(n,\C)$ is the coordinate transition function, and $A_j = g^{-1} A_i g + g^{-1} dg$.
 We first prove this for $k =0$, i.e. that there are no $R$'s in the above expression. The general case will follow by similar arguments.

 We use the following multiplicative version of the fundamental theorem of calculus for the iterated integral. For $r < s$,
\begin{equation} \label{Gltrans}
  \sum_{k \geq 0} \int_{\Delta^k_{[r,s]}} \iota(g^{-1} dg)(t_1) \dots \iota (g^{-1} dg)(t_{k}) dt_1\dots dt_k = g(r)^{-1} g(s).
\end{equation}
Here the $k$-simplex used in the integral is $\Delta^k_{[r,s]}=\{r \leq t_1 \leq \dots \leq t_k \leq s\}$. One proof of this is given by observing that the function
\[
f(s) = g(r) \left(\sum_{k \geq 0} \int_{\Delta^k_{[r,s]}} \iota(g^{-1} dg)(t_1) \dots \iota (g^{-1} dg)(t_{k}) dt_1\dots dt_k \right) g(s)^{-1}
\]
satisfies $f(r) = Id$, and also $f'(s) = 0$, since the left hand side of \eqref{Gltrans} is the formula for parallel transport for the connection $A = g^{-1}dg$ and so is the solution to the ordinary differential equation $x'(t) = x(t) \iota_{d/dt} (g^{-1} dg) $. The latter equation can also be checked by direct calculation using the fundamental theorem of calculus.

We use this to calculate 
\[
g(0)  \Bigg( \sum_{n\geq 1}\sum_{1 \leq r \leq n}  \int_{\Delta^n_{[0,1/p]}}
 \iota A_j (t_1) \dots A'_j(t_r) \dots \iota A_j (t_n) \, dt_1\dots dt_n \Bigg)  
\]
by making the substitution $A_j = g^{-1} A_i g + g^{-1} dg$ and $A'_j = g^{-1} A'_i g$, which gives
\begin{multline} \label{gIt}
g(0) \sum_{\tiny\begin{matrix}{k_1,\ldots, k_{m} \geq 0}\\{m\geq 2,0< r < m}\end{matrix}}  \int_{\Delta^{n_m+m-1}_{[0,1/p]}}
[ \iota (g^{-1} dg)(t_1) \cdots \iota (g^{-1} dg)(t_{k_1})  \wedge \iota (g^{-1} A_i g)(t_{k_1+1}) ] 
\\ \wedge 
[ \iota (g^{-1} dg)(t_{k_1+2}) \cdots \iota (g^{-1} dg)(t_{k_1+ k_2+1})  \wedge \iota (g^{-1} A_i g)(t_{k_1+ k_2+2}) ] 
\\ \wedge \cdots 
\wedge (g^{-1} A'_i g)(t_{n_r+r }) \wedge \cdots \wedge (g^{-1} A_i g)(t_{n_{m-1} + m-1}) \\ \wedge [  \iota(g^{-1} dg)(t_{n_{m-1} + m }) \cdots \iota(g^{-1} dg)(t_{n_m+m-1})] dt_1\dots dt_{n_m+m-1},
\end{multline}
where $n_i = k_1 + \dots + k_i$.

We claim this equals 
\[
\Bigg( \sum_{n\geq 1}\sum_{1 \leq r \leq n}  \int_{\Delta^n_{[0,1/p]}}
 \iota A_i (t_1) \dots A'_i(t_r) \dots \iota A_i (t_n) \, dt_1\dots dt_n \Bigg) g\left(\frac 1 p\right)
\]
To see this, for each $m$, we apply the identity in \eqref{Gltrans} $m$ times, showing all the integrals of $g^{-1}dg$ that appear, collapse.  In the first case we have
\[
g(0) \sum_{k_1 \geq 0}  \int_{\{0 \leq t_1\leq \dots \leq t_{k_1} < t_{k_1 +1} \}} \iota (g^{-1} dg)(t_1) \dots \iota (g^{-1} dg)(t_{k_1}) dt_1 \cdots dt_{k_1}= g(t_{k_1+1})
\]
reducing \eqref{gIt} to 
\begin{multline*}
\sum_{\tiny\begin{matrix}{k_2 , \ldots , k_m \geq 0}\\{m\geq 2, 0 < r < m}\end{matrix}}  \int_{\Delta^{(k_2+\dots+k_m)+m-1}_{[0,1/p]}}
\iota A_i(t_{k_1 +1}) g(t_{k_1+1}) 
\\ \wedge 
[ \iota (g^{-1} dg)(t_{k_1+2}) \cdots \iota (g^{-1} dg)(t_{k_1+ k_2+1})  \wedge \iota (g^{-1} A_i g)(t_{k_1+ k_2+2}) ] 
\\ \wedge \cdots 
\wedge (g^{-1} A'_i g)(t_{n_r+r}) \wedge \cdots \wedge (g^{-1} A_i g)(t_{n_{m-1} + m-1}) \\ \wedge [  \iota(g^{-1} dg)(t_{n_{m-1} + m }) \cdots \iota(g^{-1} dg)(t_{n_m+m-1})] dt_{k_1+1} \dots dt_{n_m+m},
\end{multline*}
where $n_i = k_1 + \dots + k_i$.

Similarly, for each $\ell$ with $1 < \ell \leq m$, fixing $t_{n_\ell + \ell}$, we have 
\begin{multline*}
\int_{ \Delta(\ell) }  
\iota A_i(t_{n_\ell + \ell}) g(t_{n_\ell + \ell}) \iota (g^{-1} dg)(t_{n_\ell + \ell+1}) \\
\dots \iota (g^{-1} dg)(t_{n_{\ell+1} + \ell}) 
dt_{n_\ell + \ell+1} \cdots dt_{n_{\ell+1} + \ell} = \iota A_i(t_{n_\ell + \ell}) g(t_{n_{\ell+1} + \ell + 1} )
\end{multline*}
where $\Delta(\ell) =  \{t_{n_\ell + \ell} \leq t_{n_\ell + \ell +1 }\leq \dots \leq t_{n_{\ell+1} + \ell} \leq t_{n_{\ell+1} + \ell+1} \}$,
and a similar formula holds for where we replace $A$ by $A'$.
Again,  $g(t_{n_{\ell+1} + \ell + 1} )$ cancels with $g^{-1}(t_{n_{\ell+1} + \ell + 1} )$ and, continuing in this way, we see that entire sum in \eqref{gIt} collapses to
\[
\Bigg( \sum_{m\geq 1}\sum_{1 \leq r \leq m}  \int_{\Delta^m_{[0,1/p]}}
 \iota A_i (t_1) \dots A'_i(t_r) \dots \iota A_i (t_m) \, dt_1\dots dt_m \Bigg) g\left(\frac 1 p\right)
\]

Similarly, the general case for $k \geq 0$ follows by using the same argument as above together with the fact that $R_j = g^{-1} R_i g$.
\end{proof}

\begin{rmk}
In \cite{TWZ}, we have used similar techniques to show that $BCh(\nabla)$ is a well defined global form on $LM$, independent of the choice of local trivialization charts $\U$, and the subdivision integer $p \in \N$. In particular this shows that
if two connections $\nabla_0$ and $\nabla_1$ are gauge equivalent, then $BCh(\nabla_0) = BCh(\nabla_1)$.
\end{rmk}

\begin{cor} \label{cor:BCSclosedforGE}
Let $\nabla_0$ and $\nabla_1$ be gauge equivalent connections on $E \to M$. Then $BCS(\nabla_s)$ is
$(d+\iota)$-closed for any path $\nabla_s$ from $\nabla_0$ to $\nabla_1$. 
\end{cor}

\begin{proof} By Theorem \ref{thm:dBCS} and the previous remark, $(d+\iota) BCS(\nabla_s) = BCh(\nabla_1) - BCh(\nabla_0) =0$.
\end{proof}

\begin{cor} \label{cor:BCSgaugeindep}
Let $\nabla_s$ be a path of connections on $E \to M$,  $g:E \to E$ be a bundle isomorphism (gauge transformation),
and let $g^*\nabla_s$ be the path of pullback connections. Then $BCS(g^*\nabla_s) = BCS(\nabla_s)$.
In partcicular, if $[BCS(\nabla_0, \nabla_1)] = 0$ then $[BCS(g^*\nabla_0, g^*\nabla_1)] = 0$ for any gauge transformation
$g$.
\end{cor}

\begin{proof} The first statement follows from the theorem since $BCS$ is well defined independent of local trivializations, and 
a global gauge transformation induces local gauge transformations. 
\end{proof}

By exactly the same argument as in Proposition \ref{prop:BCSindepchoices}, we can also prove that the form $H$ appearing in Proposition \ref{prop:bigon} is a well-defined form on $LM$.
\begin{prop}\label{PROP:H-well-def}
The locally defined form $H (\nabla_s^r)  = \sum_{k \geq  0} H_{2k+1} (\nabla_s^r)$ where $H_{2k+1} (\nabla_s^r)$ is defined by equation \eqref{Def-H_(2k+1)} is independent of choice of local trivialization charts $\U$, and the subdivision integer $p\in \N$.
\end{prop}

\begin{proof}
The proof is the same of the one in Proposition \ref{prop:BCSindepchoices}, since the only difference between $BCS(\nabla_s)$ and $H(\nabla^r_s)$ is that $BCS(\nabla_s)$ contains precisely one factor $\frac{\partial A_s}{\partial s}$ whereas $H(\nabla^r_s)$ contains one factor $\frac{\partial A^r_s}{\partial r}$ and one factor $\frac{\partial A^r_s}{\partial s}$.

We denote the dependence on the local data by $H^{(\U,p)}(\nabla^r_s)$. Then, proceeding as in Proposition \ref{prop:BCSindepchoices} as well as using the same notation, it suffices to check:
\begin{enumerate}
\item Subdivision: $H^{(\U,p)}(\nabla^r_s)=H^{(\U',r\cdot p)}(\nabla^r_s)$
\item Overlap: $H^{(\U',p)}(\nabla^r_s)(\gamma)=H^{(\U,p)}(\nabla^s_r)(\gamma)$
\end{enumerate}

For (1), we use the same argument as in Lemma \ref{lem:BCSdiv}, namely, we use equation \eqref{lem:goal} where we place instead of one factor $A'$ precisely the two factors $\frac{\partial A}{\partial r}$ and $\frac{\partial A}{\partial s}$ in all possible ways.

For (2), the same argument as in Proposition \ref{prop:BCSindepchoices} applies, now using that both partial derivatives of $A^r_s$ transform as $\frac{\partial (A^r_s)_j}{\partial r}=g^{-1}\frac{\partial (A^r_s)_i}{\partial r} g$ and $\frac{\partial (A^r_s)_j}{\partial s}=g^{-1}\frac{\partial (A^r_s)_i}{\partial s} g$.

This completes the proof the proposition.
\end{proof}

\end{document}